\documentclass[12pt]{amsart}

\pdfoutput=1

\usepackage{amssymb}
\usepackage[all]{xy}
\usepackage{eucal}
\usepackage{xfrac}
\usepackage{relsize}
\usepackage{stmaryrd}

\usepackage[backend=bibtex,
            isbn=false,
            doi=false,
            maxbibnames=4,
            firstinits=true,
            style=alphabetic,
            citestyle=alphabetic]{biblatex}
\usepackage{url}
\renewbibmacro{in:}{}
\bibliography{poissongroups.bib}

\usepackage{tikz}
\usetikzlibrary{calc,matrix}
\tikzset{w/.style={circle, draw,inner sep=1pt},b/.style={circle,draw,fill,inner sep=2pt}, s/.style={rectangle, draw,inner sep=3pt}}

\usepackage[colorlinks=true]{hyperref}
\usepackage{cleveref}

\newtheorem{thm}{Theorem}[section]
\crefname{thm}{Theorem}{Theorems}
\newtheorem{prop}[thm]{Proposition}
\crefname{prop}{Proposition}{Propositions}
\newtheorem{lm}[thm]{Lemma}
\crefname{lm}{Lemma}{Lemmas}
\newtheorem{cor}[thm]{Corollary}
\crefname{cor}{Corollary}{Corollaries}
\newtheorem{conjecture}[thm]{Conjecture}
\crefname{conjecture}{Conjecture}{Conjectures}

\crefname{question}{Question}{Questions}

\newtheorem*{theorem}{Theorem}

\theoremstyle{definition}
\newtheorem{defn}[thm]{Definition}
\crefname{defn}{Definition}{Definition}

\theoremstyle{remark}
\newtheorem{remark}[thm]{Remark}
\crefname{remark}{Remark}{Remark}
\newtheorem{example}[thm]{Example}
\crefname{example}{Example}{Example}

\newcommand{\B}{\mathrm{B}}
\newcommand{\C}{\mathrm{C}}
\renewcommand{\d}{\mathrm{d}}
\renewcommand{\H}{\mathrm{H}}
\newcommand{\N}{\mathrm{N}}
\newcommand{\T}{\mathrm{T}}
\newcommand{\U}{\mathrm{U}}
\newcommand{\rZ}{\mathrm{Z}}

\newcommand{\cA}{\mathcal{A}}
\newcommand{\cB}{\mathcal{B}}
\newcommand{\cC}{\mathcal{C}}
\newcommand{\cE}{\mathcal{E}}
\newcommand{\cG}{\mathcal{G}}
\newcommand{\cH}{\mathcal{H}}
\newcommand{\cL}{\mathcal{L}}
\newcommand{\cM}{\mathcal{M}}
\newcommand{\cO}{\mathcal{O}}
\newcommand{\cV}{\mathcal{V}}

\newcommand{\bD}{\mathbb{D}}
\newcommand{\bE}{\mathbb{E}}
\newcommand{\bL}{\mathbb{L}}
\newcommand{\bP}{\mathbb{P}}
\newcommand{\bT}{\mathbb{T}}
\newcommand{\Z}{\mathbb{Z}}

\newcommand{\hG}{\widehat{G}}

\renewcommand{\b}{\mathfrak{b}}
\newcommand{\fd}{\mathfrak{d}}
\newcommand{\g}{\mathfrak{g}}
\newcommand{\h}{\mathfrak{h}}

\newcommand{\m}{\mathfrak{m}}
\newcommand{\n}{\mathfrak{n}}
\newcommand{\p}{\mathfrak{p}}

\newcommand{\ad}{\mathrm{ad}}
\newcommand{\Ad}{\mathrm{Ad}}
\newcommand{\Alt}{\mathrm{Alt}}
\newcommand{\Arr}{\mathrm{Arr}}
\newcommand{\Bun}{\mathrm{Bun}}

\newcommand{\Cois}{\mathrm{Cois}}
\newcommand{\CYBE}{\mathrm{CYBE}}
\newcommand{\colim}{\mathrm{colim}}
\newcommand{\dR}{\mathrm{dR}}
\newcommand{\fib}{\mathrm{fib}}
\newcommand{\Hom}{\mathrm{Hom}}
\newcommand{\id}{\mathrm{id}}
\newcommand{\Lagr}{\mathrm{Lagr}}
\newcommand{\Map}{\mathrm{Map}}
\newcommand{\MC}{\underline{\mathrm{MC}}}
\newcommand{\MultBivec}{\mathrm{MultBivec}}
\newcommand{\Mod}{\mathrm{Mod}}
\newcommand{\Perf}{\mathrm{Perf}}
\newcommand{\Pois}{\mathrm{Pois}}
\newcommand{\Pol}{\mathrm{Pol}}
\newcommand{\pt}{\mathrm{pt}}
\newcommand{\QCoh}{\mathrm{QCoh}}
\newcommand{\red}{\mathrm{red}}
\newcommand{\SSet}{\mathrm{SSet}}
\newcommand{\Symp}{\mathrm{Symp}}
\newcommand{\Tot}{\mathrm{Tot}}
\newcommand{\Vect}{\mathrm{Vect}}

\newcommand{\dCE}{\mathrm{d}_{\mathrm{CE}}}
\newcommand{\ddr}{\mathrm{d}_{\mathrm{dR}}}
\newcommand{\QLie}{\mathrm{QLieBialg}}
\newcommand{\QLieAlgd}{\mathrm{QLieBialgd}}
\newcommand{\QPois}{\mathrm{QPois}}
\newcommand{\QPoisGpd}{\mathrm{QPoisGpd}}

\newcommand{\Comm}{\mathrm{Comm}}

\newcommand{\alg}{\mathrm{Alg}}
\newcommand{\balg}{\mathbf{Alg}}
\newcommand{\aug}{\mathrm{aug}}
\newcommand{\kos}{W_{\mathrm{Kos}}}
\newcommand{\Lie}{\mathrm{Lie}}
\newcommand{\coLie}{\mathrm{coLie}}
\newcommand{\coP}{\mathrm{co}\mathbb{P}}

\newcommand{\te}{\tilde{e}}
\renewcommand{\to}{\longrightarrow}

\newcommand{\defterm}[1]{\textbf{\emph{#1}}}

\newcommand{\dquot}[3]{#1\backslash #2 / #3}

\newcommand{\cosimp}[2]{
\xymatrix{
#1 \ar@<.5ex>[r] \ar@<-.5ex>[r] & #2 \ar@<.8ex>[r] \ar[r] \ar@<-.8ex>[r] & \ldots
}
}

\DeclareMathOperator{\Rep}{Rep}
\DeclareMathOperator{\Spec}{Spec}
\DeclareMathOperator{\Sym}{Sym}

\begin{document}
\title{Poisson-Lie structures as shifted Poisson structures}
\address{Department of Mathematics, University of Geneva, 2-4 rue du Lievre, 1211 Geneva, Switzerland}
\curraddr{Institut f\"{u}r Mathematik, Winterthurerstrasse 190, 8057 Z\"{u}rich, Switzerland}
\email{pavel.safronov@math.uzh.ch}
\author{Pavel Safronov}
\begin{abstract}
Classical limits of quantum groups give rise to multiplicative Poisson structures such as Poisson-Lie and quasi-Poisson structures. We relate them to the notion of a shifted Poisson structure which gives a conceptual framework for understanding classical (dynamical) $r$-matrices, quasi-Poisson groupoids and so on. We also propose a notion of a symplectic realization of shifted Poisson structures and show that Manin pairs and Manin triples give examples of such.
\end{abstract}
\maketitle

\tableofcontents
\addtocontents{toc}{\protect\setcounter{tocdepth}{1}}

\section*{Introduction}

\subsection*{Shifted Poisson geometry}

Shifted symplectic and Poisson structures introduced in \cite{PTVV} and \cite{CPTVV} respectively give a way to define higher symplectic and Poisson structures on algebraic stacks. That is, if $X$ is an algebraic stack or a derived scheme, its cotangent bundle $\T^*_X$ can be naturally enhanced to a cotangent complex $\bL_X$. In this way we can talk about differential forms and polyvectors of a nontrivial cohomological degree.

Another way to view the ``shift'' in a shifted Poisson structure is through the prism of higher deformation quantization. Given an ordinary Poisson structure on a smooth variety $X$ one can consider its deformation quantization which gives rise to a deformation of the algebra of global functions $\cO(X)$ as an associative algebra or the category of quasi-coherent sheaves $\QCoh(X)$ as a plain category (i.e. the deformation no longer has a monoidal structure). Similarly, given a $1$-shifted Poisson structure on a derived algebraic stack $X$, its deformation quantization (see \cite[Section 3.5]{CPTVV}) gives rise to a deformation of $\QCoh(X)$ as a monoidal category. More generally, a deformation quantization of an $n$-shifted Poisson stack $X$ gives rise to a deformation of $\QCoh(X)$ as an $\bE_n$-monoidal category, where $\bE_n$ is the operad of little $n$-disks.

One can also talk about Lagrangian (coisotropic) morphisms $f\colon L\rightarrow X$ if $X$ has an $n$-shifted symplectic (Poisson) structure. In contrast to the classical setting, $f$ is not required to be a closed immersion and being a Lagrangian is no longer a condition, but an extra structure. Similar to the classical setting, deformation quantization of shifted coisotropic structures (see \cite[Section 5]{MS2}) gives rise to a deformation of modules. For instance, given a morphism $f\colon L\rightarrow X$, the category $\QCoh(L)$ becomes a module category over $\QCoh(X)$. A deformation quantization of a 1-shifted coisotropic structure on $f$ gives rise to a deformation of $(\QCoh(X), \QCoh(L))$ as a pair (monoidal category, module category).

Many examples of shifted symplectic and Lagrangian structures are known. One of the goals of the present paper is to provide interesting and nontrivial examples of shifted Poisson and coisotropic structures.

\subsection*{Poisson-Lie structures}

Recall that a Poisson-Lie structure on a group $G$ is given by a multiplicative Poisson bivector $\pi\in\wedge^2 \T_G$. A more general notion is that of a quasi-Poisson group where one is given in addition a trivector $\phi\in\wedge^3(\g)$ which measures the failure of the Jacobi identity for $\pi$. In addition, given an element $\lambda\in\wedge^2(\g)$ we can twist quasi-Poisson structures, so they form a groupoid.

It is well-known that multiplicative objects on a group $G$ give rise to an object on the classifying stack $\B G = [\pt / G]$. For instance, a multiplicative line bundle on $G$ (i.e. a central extension of $G$) gives rise to a gerbe on $\B G$. Thus, it is natural to ask whether a multiplicative Poisson structure on $G$ gives rise to some structure on $\B G$.

Let $\Pois(X, n)$ be the space (i.e. an $\infty$-groupoid) of $n$-shifted Poisson structures on a derived algebraic stack $X$. The following result is given by \cref{prop:2shiftedBG} and \cref{thm:1shiftedBG}.

\begin{theorem}
Let $G$ be an algebraic group. We have the following classification results:
\begin{itemize}
\item The space $\Pois(\B G, n)$ for $n > 2$ is trivial. That is, every $n$-shifted Poisson structure on $\B G$ for $n>2$ is canonically zero.

\item The space $\Pois(\B G, 2)$ is equivalent to the set $\Sym^2(\g)^G$.

\item The space $\Pois(\B G, 1)$ is equivalent to the groupoid of quasi-Poisson structures on $G$.
\end{itemize}
\end{theorem}

These computations rely on the following trick introduced in \cite[Section 3.6.2]{CPTVV}. Let $\hG$ be the formal completion of $G$ at the unit and $G_{\dR} = G/\hG$. Then we can identify $\B G\cong (\B\hG)/G_{\dR}$. Since the cotangent complex to $G_{\dR}$ is trivial (i.e. $\B \hG\rightarrow \B G$ is formally \'{e}tale), we can compute polyvectors on $\B G$ as $G_{\dR}$-invariant polyvectors on $\B\hG$. The latter stack is an example of a formal affine stack and its polyvectors are straightforward to compute.

This classification reflects the well-known expectation that to consider a deformation quantization of $\Rep G\cong \QCoh(\B G)$ as a monoidal category one has to endow $G$ with a quasi-Poisson structure and to deformation quantize it as a braided monoidal category this structure has to be quasi-triangular, i.e. it should come from a Casimir element $c\in\Sym^2(\g)^G$.

We also have a relative analog of the previous statement. Given a morphism of derived algebraic stacks $f\colon L\rightarrow X$ we denote by $\Cois(f, n)$ the space of pairs of an $n$-shifted Poisson structure on $X$ and an $n$-shifted coisotropic structure on $L\rightarrow X$. The following is \cref{prop:coisotropicsubgroup}.

\begin{theorem}
Let $H\subset G$ be a closed subgroup. We have the following classification results:
\begin{itemize}
\item The space $\Cois(\B H\rightarrow \B G, 2)$ is equivalent to the set $\ker\left(\Sym^2(\g)^G\rightarrow \Sym^2(\g/\h)^H\right)$.

\item The space $\Cois(\B H\rightarrow \B G, 1)$ is equivalent to the groupoid of quasi-Poisson structures on $G$ for which $H$ is coisotropic.
\end{itemize}
\end{theorem}

Given a morphism $f\colon X_1\rightarrow X_2$ of algebraic stacks we denote by $\Pois(X_1\rightarrow X_2, n)$ the space of triples of $n$-shifted Poisson structures on $X_i$ and a compatibility between them making $f$ into a Poisson morphism. Then the previous statement for $H=\pt$ implies (see \cref{cor:poissongroups}) that $\Pois(\pt\rightarrow \B G, 1)$ is equivalent to the set of Poisson-Lie structures on $G$. Similar statements arise if one replaces $\B G$ by its formal completion $\B\hG$ and one replaces quasi-Poisson structures on $G$ by quasi-Lie bialgebra structures on $\g$.

Given an $n$-shifted Poisson structure on $X$ there is a natural way to extract an $(n-1)$-shifted Poisson structure on $X$. This is a classical shadow of the natural forgetful functor from $\bE_n$-monoidal categories to $\bE_{n-1}$-monoidal categories. Thus, one can ask how the above statements behave under this forgetful map.

Let us recall that if $G$ is a Poisson-Lie group, there is a dual Poisson-Lie group $G^*$ whose completion at the unit is the formal group associated with Lie algebra $\g^*$. We say it is formally linearizable if there is a Poisson isomorphism between the formal completion of $\g^*$ at the origin with the Kirillov--Kostant--Souriau Poisson structure and the formal completion of $G^*$ at the unit with its Poisson-Lie structure.

The following statement combines \cref{prop:forget21BG} and \cref{prop:forget10BG}.

\begin{theorem} $ $
\begin{itemize}
\item Suppose $c\in\Sym^2(\g)^G$ defines a $2$-shifted Poisson structure on $\B G$. Its image under $\Pois(\B G, 2)\rightarrow \Pois(\B G, 1)$ is given by the quasi-Poisson structure $(\pi=0,\phi)$ on $G$, where
\[\phi = -\frac{1}{6}[c_{12}, c_{23}].\]

\item Suppose $G$ carries a Poisson-Lie structure defining a $1$-shifted Poisson morphism $\pt\rightarrow \B\hG$. Its image under $\Pois(\pt\rightarrow \B\hG, 1)\rightarrow \Pois(\pt\rightarrow\B\hG, 0)$ is trivial iff the dual Poisson-Lie group $G^*$ is formally linearizable.
\end{itemize}
\end{theorem}

We also give an interpretation of classical (dynamical) $r$-matrices as follows. For a group $H$ the stack $[\h^*/H]\cong \T^*[1](\B H)$ has a natural 1-shifted Poisson structure. The following is \cref{prop:dynamicalrmatrix}.

\begin{theorem}
Let $H\subset G$ be a closed subgroup and $U\subset \h^*$ an $H$-invariant open subscheme. Then the space of pairs of
\begin{itemize}
\item A 2-shifted Poisson structure on $\B G$,

\item A 1-shifted Poisson structure on the composite $[U/H]\rightarrow \B H\rightarrow \B G$ compatible with the given 2-shifted Poisson structure on $\B G$ and the 1-shifted Poisson structure on $[U/H]\subset [\h^*/H]$
\end{itemize}
is equivalent to the set of quasi-triangular classical dynamical $r$-matrices with base $U$.
\end{theorem}

For instance, let us consider the case $H=\pt$ in which case we recover ordinary classical quasi-triangular $r$-matrices. Then the above statement recovers the well-known prescription that a quantization of a quasi-triangular classical $r$-matrix gives rise to a braided monoidal deformation of $\Rep G$ together with a monoidal deformation of the forgetful functor $\Rep G\rightarrow \Vect$.

To explain the general case, let us first observe that $\QCoh([\h^*/H])\cong \Mod_{\Sym(\h)}(\Rep H)$. Its monoidal deformation quantization is given by the monoidal category
\[\cH_H = \Mod_{\U(\h)}(\Rep H)\]
equivalent to the so-called category of Harish-Chandra bimodules, i.e. $\U(\h)$-bimodules whose diagonal $\h$-action is $H$-integrable. Then a deformation quantization of a classical dynamical $r$-matrix with base $\h^*$ is given by a braided monoidal deformation of $\Rep G$ together with a monoidal functor $\Rep G\rightarrow \cH_H$. If the dynamical $r$-matrix has poles (i.e. $U\neq \h^*$), then this will be a lax monoidal functor which is only generically (i.e. over $U$) monoidal.

This theorem explains the seemingly different geometric interpretations of dynamical $r$-matrices in terms of quasi-Poisson $G$-spaces (see \cite[Section 2.2]{EE}) and dynamical Poisson groupoids (see \cite{EV1}). Indeed, consider the $G$-space $Y = U\times G$. Then $[Y/G]\rightarrow \B G$ is equivalent to $U\rightarrow \B G$. We may also consider the groupoid $\cG = U\times G\times U\rightrightarrows U$. Then $U\rightarrow [U/\cG]$ is also equivalent to $U\rightarrow \B G$.

We show how some standard (dynamical) $r$-matrices can be constructed from Lagrangian correspondences in \cref{sect:rmatrixconstruction}.

\subsection*{Quasi-Poisson groupoids}

The description of quasi-Poisson structures on a group in terms of 1-shifted Poisson structures on its classifying stack can be generalized to groupoids. Namely, let us consider a source-connected smooth affine groupoid $\cG\rightrightarrows X$ over a smooth affine scheme. The following is \cref{thm:quasipoissongroupoids}.

\begin{theorem}
The space of 1-shifted Poisson structures on $[X/\cG]$ is equivalent to the groupoid of quasi-Poisson structures on $\cG$ with morphisms given by twists.
\end{theorem}

Let us note that in \cite{BCLGX} quasi-Poisson groupoids are also described in terms of Maurer--Cartan elements. However, the two approaches are opposite. In \cite{BCLGX} the authors show that the notion of a quasi-Poisson groupoid is Morita-invariant and so it defines a geometric structure on the stack $[X/\cG]$ independent of the presentation. In this paper we start with a manifestly Morita-invariant notion of a 1-shifted Poisson structure on $[X/\cG]$ and show that it is equivalent to a quasi-Poisson structure on $\cG$ thus showing that the latter is Morita-invariant.

A much easier computation can be performed for 1-shifted symplectic structures (see \cref{prop:quasisymplecticgroupoid}).

\begin{theorem}
The space of 1-shifted symplectic structures on $[X/\cG]$ is equivalent to the groupoid of quasi-symplectic structures on $\cG$ with morphisms given by twists.
\end{theorem}

One may identify 1-shifted symplectic structures on a derived stack with the space of \emph{nondegenerate} 1-shifted Poisson structures (see \cite[Theorem 3.2.4]{CPTVV} and \cite[Theorem 3.33]{Pri1}). Thus, the groupoid of quasi-symplectic structures on $\cG$ is identified with the groupoid of nondegenerate quasi-Poisson structures on $\cG$. Explicitly, it means the following. A quasi-Poisson structure on $\cG$ gives rise to a quasi-Lie bialgebroid structure on its underlying Lie algebroid $\cL$. In turn, a quasi-Lie bialgebroid $\cL$ gives rise to a Courant algebroid $\cL\oplus \cL^*$ (see \cite{Ro}). The 1-shifted Poisson structure on $[X/\cG]$ is nondegenerate iff the Courant algebroid $\cL\oplus \cL^*$ is exact.

\subsection*{Symplectic realizations}

Given a Poisson manifold $X$ the cotangent bundle $\T^*_X$ becomes a Lie algebroid with respect to the so-called Koszul bracket; moreover, this Lie algebroid has a compatible symplectic structure. Thus, one might ask if it integrates to a symplectic groupoid $\cG\rightrightarrows X$ which is a groupoid equipped with a multiplicative symplectic structure on $\cG$. Conversely, given such a symplectic groupoid we get an induced Poisson structure on $X$. Symplectic groupoids give symplectic realizations \cite{We} of Poisson manifolds, thus one might ask for a similar notion in the setting of shifted symplectic structures. We warn the reader that what we only restrict to symplectic realizations which are symplectic groupoids (see \cref{remark:symplecticrealization}).

Given a symplectic groupoid $\cG\rightrightarrows X$ we have an induced $1$-shifted symplectic structure on $[X/\cG]$ together with a Lagrangian structure on the projection $X\rightarrow [X/\cG]$ (see \cref{prop:symplecticgroupoid}). Now, given any $n$-shifted Lagrangian $L\rightarrow Y$ by the results of \cite[Section 4]{MS2} we get an induced $(n-1)$-shifted Poisson structure on $L$. In the case of the 1-shifted Lagrangian $X\rightarrow [X/\cG]$ we get an unshifted Poisson structure on $X$ coming from the symplectic groupoid.

Thus, we define symplectic realizations of $n$-shifted Poisson stacks $X$ to be lifts of those to $(n+1)$-shifted Lagrangians $X\rightarrow Y$. It is expected that any $n$-shifted Poisson stack $X$ has a unique formal symplectic realization, i.e. a symplectic realization $X\rightarrow Y$ which is an equivalence of reduced stacks (a nil-isomorphism). We refer to \cite[Section 3]{Cal2} for a discussion of this. The work \cite{Sp} in fact uses this as a definition of $n$-shifted Poisson structures.

We illustrate symplectic realizations by showing that the Feigin--Odesskii \cite{FO} Poisson structure on $\Bun_P(E)$, the moduli space of $P$-bundles on an elliptic curve $E$ for $P$ a parabolic subgroup of a simple group $G$, admits a symplectic realization given by the 1-shifted Lagrangian $\Bun_P(E)\rightarrow \Bun_M(E)\times \Bun_G(E)$, where $M$ is the Levi factor of $P$. In particular, by taking the \v{C}ech nerve of this map we recover a symplectic groupoid integrating the Feigin--Odesskii Poisson structure.

Recall that a Manin pair is a pair $\g\subset \fd$ of Lie algebras where $\fd$ is equipped with an  invariant nondegenerate pairing and $\g\subset \fd$ is Lagrangian. Suppose that the Manin pair $\g\subset \fd$ integrates to a group pair $G\subset D$. It is known that it induces a quasi-Poisson structure on $G$. Given a quasi-Poisson group $G$ we get an induced $1$-shifted Poisson structure on $\B G$ and we show that a Manin pair gives its symplectic realization which is a 2-shifted Lagrangian $\B G\rightarrow \B D$.

Similarly, given a shifted coisotropic morphism $C\rightarrow X$ we propose a notion of symplectic realizations for those which are given by $(n+1)$-shifted Lagrangian correspondences
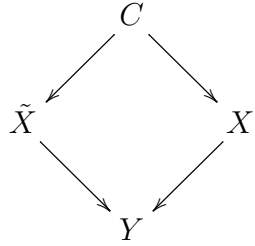
\[
\xymatrix{
& C \ar[dl] \ar[dr] & \\
\tilde{X} \ar[dr] && X \ar[dl] \\
& Y &
}
\]
where $Y$ carries an $(n+1)$-shifted symplectic structure, $X\rightarrow Y$ and $\tilde{X}\rightarrow Y$ are Lagrangian and so is $C\rightarrow \tilde{X}\times_Y X$. Assuming a certain conjecture (\cref{conj:intersections}) on a compatibility between Lagrangian and coisotropic intersections, we see that $C\rightarrow X$ inherits an $n$-shifted coisotropic structure.

Recall that a Manin triple is a triple $(\fd, \g, \g^*)$ where $\g\subset \fd$ and $\g^*\subset \fd$ are Manin pairs and $\g$ and $\g^*$ intersect transversely. Again suppose the Manin triple integrates to a triple of groups $(D, G, G^*)$. It is known that it induces a Poisson-Lie structure on $G$.

Given a Manin triple $(D, G, G^*)$ we obtain a $2$-shifted Lagrangian correspondence
\[
\xymatrix{
& \pt \ar[dl] \ar[dr] & \\
\B G^* \ar[dr] && \B G \ar[dl] \\
& \B D
}
\]

In particular, $\pt\rightarrow \B G$ and $\pt\rightarrow \B G^*$ carry a $1$-shifted coisotropic structure and hence $G$ and $G^*$ become Poisson-Lie groups, thus Manin triples give symplectic realizations of Poisson-Lie structures.

\subsection*{Acknowledgments}

The author would like to thank A. Brochier, D. Calaque, D. Jordan, C. Laurent-Gengoux and B. Pym for useful conversations. This research was supported by the NCCR SwissMAP grant of the Swiss National Science Foundation.

\addtocontents{toc}{\protect\setcounter{tocdepth}{2}}

\section{Poisson and coisotropic structures on stacks}

In this section we remind the reader the necessary basics of shifted Poisson and shifted coisotropic structures on derived stacks as defined in \cite{CPTVV}, \cite{MS1} and \cite{MS2}.

\subsection{Formal geometry}

Recall the notion of a graded mixed cdga from \cite[Section 1.1, Section 1.5]{CPTVV} and the ind-object $k(\infty)$. Given a graded object $A$ we denote $A(\infty) = A\otimes k(\infty)$. Explicitly, a graded mixed cdga is a cdga $A$ together with an extra \defterm{weight} grading and a square-zero derivation $\epsilon$ of degree $1$ and weight $1$. Given a graded mixed cdga $A$ its realization is
\[|A|\cong \prod_{n\geq 0} A(n)\]
with the differential $\d_A+\epsilon$ and its Tate realization is
\[|A|^t\cong \underset{m\rightarrow \infty}{\colim}\ \prod_{n\geq -m} A(n)\]
with the same differential. We can also identify
\[|A|^t\cong |A\otimes k(\infty)|.\]

Here are two important examples that we will use in this paper:
\begin{itemize}
\item Let $X$ be a derived Artin stack. Then one has the graded mixed cdga $\Omega^\epsilon(X)$ of differential forms on $X$ (see \cite{PTVV} where it is denoted by $\mathbf{DR}(X)$). As a graded cdga it can be identified with
\[\Omega^\epsilon(X)\cong \Gamma(X, \Sym(\bL_X[-1]))\]
where the weight of $\bL_X$ is $1$. One can think of the mixed structure as the de Rham differential $\ddr$. We denote by $\Omega^\bullet(X)$ its realization.

\item Let $\g$ be a Lie algebra and $A$ a commutative algebra with a $\g$-action. Then one can define
\[\C^\epsilon(\g, A)\cong \Hom(\Sym(\g[1]), A)\]
as a graded cdga with the mixed structure given by the Chevalley--Eilenberg differential $\dCE$. We denote by $\C^\bullet(\g, A)$ its realization.
\end{itemize}

Let $\g$ be a Lie algebra and $\hG$ the corresponding formal group. Throughout the paper we will be interested in the stack
\[\B\g = \B\hG.\]
It is easy to see that $\cO(\B\g)$ coincides with the Lie algebra cohomology of $\g$ which follows for instance from \cite[Theorem 2.4.1]{DAGX}. Even though $\B\g$ is not affine, many of its properties are essentially determined by its algebra of functions $\cO(\B\g)$, more precisely by its variant $\bD(\B\g)$ that we will construct shortly.

For an affine scheme $S$ we define the graded mixed cdga $\bD(S)$ to be
\[\bD(S) = \Omega^\epsilon(S_{\red}/S).\]
For a general stack $X$ we define
\[\bD(X) = \lim_{S\rightarrow X}\bD(S),\]
where the limit is over affine schemes $S$ mapping to $X$. We can identify
\[|\bD(X)|\cong \cO(X).\]

Given a morphism $X\rightarrow Y$ of stacks we define the relative de Rham space to be
\[(X/Y)_{\dR} = X_{\dR}\times_{Y_{\dR}} Y.\]

\begin{prop}
Let $X\rightarrow Y$ be a morphism of affine schemes where $X$ is reduced. Then we have an equivalence of graded mixed cdgas
\[\bD((X/Y)_{\dR})\cong \Omega^\epsilon(X/Y).\]
\label{prop:deRhamfunctions}
\end{prop}
\begin{proof}
Let us begin by constructing a morphism $\Omega^\epsilon(X/Y)\rightarrow \bD((X/Y)_{\dR})$ of graded mixed cdgas. Given $\Spec A\rightarrow (X/Y)_{\dR}$ we have a commutative diagram
\[
\xymatrix{
\Spec A_{\red} \ar[d] \ar[r] & \Spec A \ar[d] & \\
X \ar[r] & (X/Y)_{\dR} \ar[dr] & \\
&& Y
}
\]

Moreover, we have pullbacks
\[\Omega^\epsilon(X/Y)\to \Omega^\epsilon(A_{\red}/Y)\to \Omega^\epsilon(A_{\red}/A)= \bD(A).\]
This is compatible with pullbacks along $A$ and hence we obtain a morphism of graded mixed cdgas
\[\Omega^\epsilon(X/Y)\to \bD((X/Y)_{\dR}) = \lim_{\Spec A\rightarrow (X/Y)_{\dR}} \bD(A).\]

Since $((X/Y)_{\dR})_{\red}\cong X$, the underlying morphism of graded cdgas is an equivalence by \cite[Proposition 2.2.7]{CPTVV}.
\end{proof}

Given a stack $X$ we have its de Rham space $X_{\dR}$ which has the functor of points
\[X_{\dR}(R) = X(\H^0(R)_{red}).\]
For a smooth scheme $X$, its de Rham space can be constructed as the quotient of $X$ by its infinitesimal groupoid. Now let $X$ be a derived Artin stack and consider $p\colon X\rightarrow X_{\dR}$. Then $p_*\cO_X$ is a commutative $\cO_{X_{\dR}}$-algebra. Following \cite[Definition 2.4.11]{CPTVV}, one can enhance the pair $(p_*\cO_X, \cO_{X_{\dR}})$ to prestacks of graded mixed cdgas over $X_{\dR}$ denoted by $\cB_X = \bD_{X/X_{\dR}}$ and $\bD_{X_{\dR}}$ respectively so that $\cB_X$ is a $\bD_{X_{\dR}}$-cdga. Moreover, we have equivalences of prestacks
\[|\cB_X|\cong p_*\cO_X,\qquad |\bD_{X_{\dR}}|\cong \cO_{X_{\dR}}.\]

\subsection{Maurer--Cartan spaces}

Let us recall the necessary results about Maurer--Cartan spaces. Let $\g$ be a nilpotent dg Lie algebra and $\Omega_\bullet=\Omega^\bullet(\Delta^n)$ the simplicial algebra of polynomial differential forms on simplices. We define $\MC(\g)$ to be the simplicial set of Maurer--Cartan elements in $\g\otimes\Omega_\bullet$. More generally, if $\g$ is a pro-nilpotent dg Lie algebra, we define $\MC(\g)$ to be the inverse limit of Maurer--Cartan spaces of the filtration.

We will use the following useful way to compute Maurer--Cartan spaces \cite[Proposition 2.2.3]{Hin}:
\begin{prop}
Suppose $\g$ is a nilpotent dg Lie algebra concentrated in non-negative degrees. Then $\MC(\g)$ is equivalent to the nerve of the following \defterm{Deligne groupoid}:
\begin{itemize}
\item Its objects are Maurer--Cartan elements in $\g$.

\item Its morphisms from $x$ to $y$ are given by elements $\lambda\in\g^0$ and a Maurer--Cartan element $\alpha(t)\in\g\otimes k[t]$ satisfying the following equations:
\begin{align*}
\frac{d \alpha(t)}{dt} + \d\lambda + [\alpha(t), \lambda] &= 0 \\
\alpha(0) &= x \\
\alpha(1) &= y.
\end{align*}
\end{itemize}
\label{prop:MCDeligne}
\end{prop}

If $\g$ is a graded dg Lie algebra with a bracket of weight $-1$, its completion in weights $\geq 2$ denoted by $\g^{\geq 2}$ is a pro-nilpotent dg Lie algebra. Moreover, suppose $\g_\bullet$ is a cosimplicial graded dg Lie algebra. It follows from \cite[Proposition 1.17]{MS1} that the functor $\MC((-)^{\geq 2})\colon \alg_{\Lie}^{gr}\rightarrow \SSet$ preserves homotopy limits, so we obtain the following statement. See also \cite[Theorem 4.1]{Hin} and \cite[Theorem 3.11]{Ban} for a closely related statement.

\begin{prop}
We have an equivalence of spaces
\[\MC(\Tot(\g_\bullet)^{\geq 2})\cong \Tot(\MC(\g_\bullet^{\geq 2})).\]
\label{prop:MClimits}
\end{prop}

We will also use the following lemma to compute totalizations of cosimplicial groupoids (see \cite[Corollary 2.11]{Ho1}):
\begin{lm}
Let $\cG^\bullet$ be a cosimplicial groupoid
\[\cosimp{\cG^0}{\cG^1}\]
Its totalization is equivalent to the following groupoid $\cG$:
\begin{itemize}
\item Objects of $\cG$ are objects $a$ of $\cG^0$ together with an isomorphism $\alpha\colon \d^1(x)\rightarrow \d^0(x)$ in $\cG^1$ satisfying $s^0(\alpha)=\id_a$ and $\d^0(\alpha)\circ \d^2(\alpha)=\d^1(\alpha)$.

\item Morphisms in $\cG$ from $(a, \alpha)$ to $(a', \alpha')$ are given by isomorphisms $\beta\colon a\rightarrow a'$ in $\cG^0$ such that
\[
\xymatrix{
\d^1(a) \ar^{\d^1(\beta)}[r] \ar^{\alpha}[d] & \d^1(a') \ar^{\alpha'}[d] \\
\d^0(a) \ar^{\d^0(\beta)}[r] & \d^0(a')
}
\]
commutes.
\end{itemize}
\label{lm:groupoidTot}
\end{lm}

\subsection{Poisson and coisotropic structures}

Recall that the operad $\bP_n$ is a dg operad controlling commutative dg algebras with a degree $1-n$ Poisson bracket. Given an $\infty$-category $\cC$ we denote by $\cC^{\sim}$ the underlying $\infty$-groupoid of objects.

\begin{defn}
Let $A$ be a commutative dg algebra. The \defterm{space of $n$-shifted Poisson structures} $\Pois(A, n)$ is defined to be the homotopy fiber of the forgetful map
\[\balg_{\bP_{n+1}}^{\sim}\to \balg_{\Comm}^{\sim}\]
at $A\in\balg_{\Comm}$.
\end{defn}

For instance, if $A$ is a smooth commutative algebra, the space $\Pois(A, 0)$ is discrete and given by the set of Poisson structures on $\Spec A$.

\begin{defn}
Let $X$ be a derived Artin stack. The \defterm{space of $n$-shifted Poisson structures} $\Pois(X, n)$ is defined to be the space of $n$-shifted Poisson structures on $\cB_X(\infty)$ as a $\bD_{X_{\dR}}(\infty)$-algebra.
\end{defn}

Let us briefly explain how to think about the twist $k(\infty)$ in the definition. Suppose $A$ is a graded mixed commutative algebra. Then a compatible strict $\bP_{n+1}$-algebra structure on $A$ is given by a Poisson bracket of weight $0$ such that the mixed structure is a biderivation. We can also weaken the compatibility with the mixed structure so that a weak $\bP_{n+1}$-algebra structure on $A$ would be given by a sequence of Poisson brackets $\{[-, -]_n\}_{n\geq 0}$ of weight $n$ which are biderivations with respect to the total differential $(\d+\epsilon)$.

Similarly, a weak $\bP_{n+1}$-algebra structure on $A(\infty)$ is given by a sequence of Poisson brackets $\{[-, -]_n\}_{n\in\Z}$ such that for any fixed elements $x,y\in A$ the expressions $[x, y]_n$ are zero for negative enough $n$.

We define coisotropic structures as follows. Consider the colored operad $\bP_{[n+1, n]}$ whose algebras are triples $(A, B, F)$, where $A$ is a $\bP_{n+1}$-algebra, $B$ is a $\bP_n$-algebra and $F\colon A\rightarrow \rZ(B)$ is a $\bP_{n+1}$-morphism, where
\[\rZ(B) = \Hom_B(\Sym_B(\Omega^1_B[n]), B)\]
is the \defterm{Poisson center} with the differential twisted by $[\pi_B, -]$. Given such a $\bP_{[n+1, n]}$-algebra, the composite $A\rightarrow \rZ(B)\rightarrow B$ is a morphism of commutative algebras which gives a forgetful functor
\[\balg_{\bP_{[n+1, n]}}\to \Arr(\balg_{\Comm}).\]

\begin{defn}
Let $f\colon A\rightarrow B$ be a morphism of commutative dg algebras. The \defterm{space of $n$-shifted coisotropic structures} $\Cois(f, n)$ is defined to be the homotopy fiber of the forgetful map
\[\balg_{\bP_{[n+1, n]}}^{\sim}\to \Arr(\balg_{\Comm})^{\sim}\]
at $f\in\Arr(\balg_{\Comm})$.
\end{defn}

Suppose $(A, B, F)$ is a $\bP_{[n+1, n]}$-algebra and let $f\colon A\rightarrow B$ be the induced morphism of commutative algebras. If we denote the homotopy fiber of $f$ by $\U(A, B)$, it is shown in \cite[Section 3.5]{MS1} that $\U(A, B)[n]$ has a natural dg Lie algebra structure such that
\[B[n-1]\to \U(A, B)[n]\to A[n]\]
becomes a fiber sequence of Lie algebras. Moreover, if $A\rightarrow B$ is surjective we can identify $\U(A, B)$ with the strict kernel of $A\rightarrow B$ with the Lie bracket induced from $A$.

Suppose $f\colon L\rightarrow X$ is a morphism of derived Artin stacks. We denote the induced morphism on de Rham spaces by
\[f_{\dR}\colon L_{\dR}\to X_{\dR}.\]
Moreover, we get a pullback morphism $f_{\dR}^*\cB_X\rightarrow \cB_L$ of $\bD_{L_{\dR}}$-algebras. Now, if $X$ in addition has an $n$-shifted Poisson structure, we obtain a natural $n$-shifted Poisson structure on $f_{\dR}^*\cB_X(\infty)$.

\begin{defn}
Let $f\colon L\rightarrow X$ a morphism of derived Artin stacks. The \defterm{space of $n$-shifted coisotropic structures} $\Cois(f, n)$ is defined to be the space of pairs $(\gamma_L, \pi_X)$ of an $n$-shifted coisotropic structure $\gamma_L$ on $f_{\dR}^*\cB_X(\infty)\rightarrow \cB_L(\infty)$ as $\bD_{L_{\dR}}(\infty)$-algebras and an $n$-shifted Poisson structure $\pi_X$ on $X$ such that the induced $n$-shifted Poisson structures on $f_{\dR}^*\cB_X(\infty)$ coincide.
\end{defn}

Note that the space $\Cois(f, n)$ in particular contains the information of an $n$-shifted Poisson structure on $X$ so that we get a diagram of spaces
\[
\xymatrix{
& \Cois(f, n) \ar[dl] \ar[dr] & \\
\Pois(L, n - 1) && \Pois(X, n)
}
\]

For instance, consider the identity morphism $\id\colon X\rightarrow X$. By \cite[Proposition 4.16]{MS1} the morphism $\Cois(\id, n)\rightarrow \Pois(X, n)$ is an equivalence, so we obtain a forgetful map
\[\Pois(X, n)\cong \Cois(\id, n)\rightarrow \Pois(X, n-1)\]
allowing us to reduce the shift. We will analyze this map in some examples in \cref{sect:forgetshift}.

If $f\colon Y\rightarrow X$ is a morphism of derived Artin stacks, one can define the \defterm{space of $n$-shifted Poisson morphisms} $\Pois(f, n)$ as the space of compatible pairs of $n$-shifted Poisson structures on $X$ and $Y$, we refer to \cite[Definition 2.8]{MS2} for a precise definition. The following statement (\cite[Theorem 2.8]{MS2}) allows us to reduce the computation of the space of Poisson morphisms to the space of coisotropic morphisms:

\begin{thm}
Let $f\colon Y\rightarrow X$ be a morphism of derived Artin stacks and denote by $g\colon Y\rightarrow Y\times X$ its graph. One has a Cartesian diagram of spaces
\[
\xymatrix{
\Pois(f, n) \ar[r] \ar[d] & \Pois(Y, n) \times \Pois(X, n) \ar[d] \\
\Cois(g, n) \ar[r] & \Pois(Y\times X, n)
}
\]
\label{thm:graphcoisotropic}
\end{thm}

Note that in the diagram the morphism $\Pois(Y, n)\times \Pois(X, n)\rightarrow \Pois(Y\times X, n)$ is given by sending $(\pi_Y, \pi_X)\mapsto \pi_Y - \pi_X$.

This is a generalization of the classical statement that a morphism of Poisson manifolds $Y\rightarrow X$ is Poisson iff its graph is coisotropic.

\subsection{Polyvectors}

Let us briefly sketch a computationally-efficient way of describing Poisson and coisotropic structures in terms of Maurer--Cartan spaces.

\begin{defn}
Let $A$ be a commutative dg algebra. We define the \defterm{algebra of $n$-shifted polyvector fields} $\Pol(A, n)$ to be the graded $\bP_{n+2}$-algebra
\[\Pol(A, n) = \Hom_A(\Sym_A(\Omega^1_A[n+1]), A)\]
with the Poisson bracket given by the Schouten bracket of polyvector fields.
\end{defn}

If $A$ is a graded mixed cdga, one can consider two variants. First, let $\Pol^{int}(A, n)$ be the bigraded mixed $\bP_{n+2}$-algebra defined as above. Then we define
\begin{align*}
\widetilde{\Pol}(A, n) &= |\Pol^{int}(A, n)| \\
\Pol(A, n) &= |\Pol^{int}(A, n)|^t.
\end{align*}
We will only consider the latter variant of polyvector fields in this paper, see \cref{remark:tate} for the difference.

If $X$ is a derived stack, $\cB_X$ is a prestack of $\bD_{X_{\dR}}$-linear graded mixed cdgas on $X_{\dR}$ and hence $\Pol(\cB_X, n)$ is a prestack of graded $\bP_{n+2}$-algebras on $X_{\dR}$. We define
\[\Pol(X, n) = \Gamma(X_{\dR}, \Pol(\cB_X, n)).\]
Let us also denote by $\Pol(X, n)^{\geq 2}$ the completion of this graded dg Lie algebra in weights $\geq 2$. The following is \cite[Theorem 3.1.2]{CPTVV}.

\begin{thm}
Let $X$ be a derived Artin stack. Then one has an equivalence of spaces
\[\Pois(X, n)\cong \MC(\Pol(X, n)^{\geq 2}[n+1]).\]
\label{thm:poissonpolyvectors}
\end{thm}

Similarly, if $f\colon L\rightarrow X$ is a morphism of derived Artin stacks, one can define the relative algebra of polyvectors $\Pol(L/X, n-1)$ which is a graded $\bP_{n+1}$-algebra. The following statement shows that we have control over $\Pol(X, n)$ as a graded cdga (but not as a Lie algebra).

\begin{prop}
One has equivalences of graded cdgas
\[\Pol(X, n)\cong \Gamma(X, \Sym(\bT_X[-n-1]))\]
and
\[\Pol(L/X, n-1)\cong \Gamma(L, \Sym(\bT_{L/X}[-n])).\]
\end{prop}

Observe that we have a morphism of graded cdgas
\[\Pol(X, n)\to \Pol(L/X, n-1)\]
induced by the natural morphism $f^*\bT_X\rightarrow \bT_{L/X}[1]$. It is shown in \cite[Section 2.2]{MS2} that we can upgrade the pair $(\Pol(X, n), \Pol(L/X, n-1))$ to a graded $\bP_{[n+2, n+1]}$-algebra; denote the \defterm{algebra of relative $n$-shifted polyvectors} by
\[\Pol(f, n) = \U(\Pol(X, n), \Pol(L/X, n-1)).\]

\begin{thm}
Let $f\colon L\rightarrow X$ be a morphism of derived Artin stacks. Then one has an equivalence of spaces
\[\Cois(f, n)\cong \MC(\Pol(f, n)^{\geq 2}[n+1]).\]
\label{thm:coisotropicrelpolyvectors}
\end{thm}

\subsection{The case of a classifying stack}

\label{sect:BGpolyvectors}

Let $G$ be an affine algebraic group and denote by $\hG$ its formal completion at the unit. $\hG\subset G$ is a normal subgroup and hence $\B\hG$ carries an action of $[G/\hG]\cong G_{\dR}$.

Identifying $\g$ with right-invariant vector fields we get an isomorphism of graded cdgas
\[\Omega^\epsilon(G)\cong \C^\epsilon(\g, \cO(G)).\]

\begin{lm}
The isomorphism
\[\Omega^\epsilon(G)\cong \C^\epsilon(\g, \cO(G))\]
is compatible with the mixed structures, where the $\g$-action on $\cO(G)$ is given by infinitesimal left translations.
\label{lm:CEdeRham}
\end{lm}

The graded mixed cdga $\C^\epsilon(\g, \cO(G))$ has a bialgebra structure transferred from $\Omega^\epsilon(G)$ whose coproduct is uniquely determined by the following properties:
\begin{itemize}
\item The diagram
\[
\xymatrix{
\cO(G) \ar^-{\Delta}[r] \ar[d] & \cO(G)\otimes \cO(G) \ar[d] \\
\Omega^\epsilon(G) \ar^-{\Delta}[r] & \Omega^\epsilon(G)\otimes \Omega^\epsilon(G)
}
\]
is commutative.

\item For $\alpha\in\g^*\subset \C^\epsilon(\g, \cO(G))$ we have
\[\Delta(\alpha) = \alpha\otimes 1 + \Ad_{g_{(1)}} (1\otimes \alpha),\]
where $g_{(1)}$ is the coordinate on the first factor of $\cO(G)$ in $\Omega^\epsilon(G)\otimes \Omega^\epsilon(G)$.
\end{itemize}

Suppose $V$ is a representation of $G$. Then $V^{\hG}=\C^\epsilon(\g, V)$ is a representation of $G_{\dR}$, i.e. it carries a coaction of $\C^\epsilon(\g, \cO(G))$ with the formulas generalizing the ones above.

The following statement was proved in \cite[Proposition 3.6.3]{CPTVV} by a different method.
\begin{thm}
One has an equivalence of graded mixed cdgas
\[\bD(\B\g)\cong \C^\epsilon(\g, k)\]
compatible with the actions of $G_{\dR}$ on both sides.
\label{thm:DBg}
\end{thm}
\begin{proof}
Since
\[\B\g\cong [G_{\dR} / G],\]
we have
\[\bD(\B\g)\cong \bD(G_{\dR})^G,\]
where we regard $\bD(G_{\dR})$ as a $\bD(G)$-comodule and take the corresponding derived invariants. Since $G$ is reduced, $\bD(G)\cong \cO(G)$. Moreover, by \cref{prop:deRhamfunctions} we get $\bD(G_{\dR})\cong \Omega^\epsilon(G)$. Therefore,
\[\bD(\B\g)\cong \Omega^\bullet(G)^G.\]
By \cref{lm:CEdeRham} we identify
\[\Omega^\bullet(G)\cong \C^\epsilon(\g, \cO(G))\]
as graded mixed cdgas and hence
\[\bD(\B\g)\cong \C^\epsilon(\g, k).\]
\end{proof}

Recall that $\C^\bullet(\g, \Sym(\g[-n]))$ is naturally a $\bP_{n+2}$-algebra with the so-called \emph{big bracket} \cite{KSch} given by contracting $\g^*$ and $\g$. Moreover, it is naturally graded if we assign weight $1$ to $\g[-n]$.

\begin{prop}
One has an equivalence of graded $\bP_{n+2}$-algebras
\[\Pol(\B\g, n)\cong \C^\bullet(\g, \Sym(\g[-n])).\]
\label{prop:polyvectorsBg}
\end{prop}
\begin{proof}
We have $\bT_{\bD(\B\g)}\cong \C^\epsilon(\g, \g[1])$, where $\g[1]$ is in internal weight $-1$. Therefore,
\[\Pol(\B\g, n)\cong |\Sym_{\bD(\B\g)}(\bT_{\bD(\B\g)}[-n-1])|^t\cong \C^\bullet(\g, \Sym(\g[-n])).\]
\end{proof}

\begin{remark}
Note that from the proof we see that the answer is different if we do not take the Tate realization. In that case the weight $p$ part of $\widetilde{\Pol}(\B\g, n)$ is
\[\C^{\geq p}(\g, \Sym^p(\g[-n]))\]
while the weight $p$ part of $\Pol(\B\g, n)$ is
\[\C^\bullet(\g, \Sym^p(\g[-n])).\]
\label{remark:tate}
\end{remark}

$\Pol(\B\g, n)$ has a natural action of $\Omega^\bullet(G)\cong \cO(G_{\dR})$ given by \cref{lm:CEdeRham} and we have
\[\Pol(\B G, n)\cong \Pol(\B\g, n)^{G_{\dR}}.\]

Now suppose $H\subset G$ is a closed subgroup with Lie algebra $\h\subset \g$. Let $f\colon \B H\rightarrow \B G$ be the induced map on classifying stacks and $\hat{f}\colon \B\h\rightarrow \B\g$ be its completion. Since $\bT_{\B\h/\B\g}\cong \g/\h$ as $\h$-representations, we have a quasi-isomorphism of graded mixed \emph{cdgas}
\[\Pol(\B \h/\B \g, n-1)\cong \C^\bullet(\h, \Sym(\g/\h[-n])).\]
Since the projection $\Pol(\B\g, n)\rightarrow \Pol(\B\h/\B\g, n-1)$ is surjective, by \cite[Proposition 4.10]{MS1} we can identify
\[\Pol(\hat{f}, n)\cong \ker(\Pol(\B\g, n)\to \Pol(\B\h/\B\g, n-1)),\]
where the Lie bracket is induced from the one on $\Pol(\B\g, n)$.

Finally, as before we identify
\[\Pol(f, n)\cong \Pol(\hat{f}, n)^{H_{\dR}}.\]

\section{Poisson-Lie groups}

In this section we show that many classical notions from Poisson-Lie theory have a natural interpretation in terms of shifted Poisson structures.

\subsection{Classical notions}

\label{sect:classicalpoissonlie}

We begin by recalling some standard notions from Poisson-Lie theory. Let $G$ be an algebraic group and $\g$ its Lie algebra. Recall from \cref{sect:BGpolyvectors} that $\C^\bullet(\g, \Sym(\g[-1]))$ is naturally a $\bP_3$-algebra.

\begin{defn}
A \defterm{quasi-Lie bialgebra structure} on $\g$ is the data of $\delta\colon \g\rightarrow \wedge^2(\g)$ and $\phi\in\wedge^3(\g)$ satisfying the following equations:
\begin{align}
\dCE \delta &= 0 \label{eq:qlie1} \\
\frac{1}{2} [\delta, \delta] + \dCE \phi &= 0 \label{eq:qlie2} \\
[\delta, \phi] &= 0 \label{eq:qlie3}.
\end{align}

A \defterm{Lie bialgebra structure} on $\g$ is a quasi-Lie bialgebra structure with $\phi=0$.
\end{defn}

Denote the natural $\bP_2$ bracket on $\Sym(\g[-1])$ by $\llbracket-, -\rrbracket$. The relation between this bracket and the big $\bP_3$ bracket on $\C^\bullet(\g, \Sym(\g[-1]))$ is given by
\[\llbracket x, x\rrbracket = [x, \d x],\qquad x\in\Sym(\g[-1]).\]

\begin{defn}
Let $\g$ be a Lie algebra. The \defterm{groupoid $\QLie(\g)$ of quasi-Lie bialgebra structures on $\g$} is defined as follows:
\begin{itemize}
\item Objects of $\QLie(\g)$ are quasi-Lie bialgebra structures on $\g$.

\item Morphisms of $\QLie(\g)$ from $(\delta_1, \phi_1)$ to $(\delta_2, \phi_2)$ are given by $\lambda\in\wedge^2(\g)$ satisfying
\begin{align*}
\delta_2 &= \delta_1 + \dCE \lambda \\
\phi_2 &= \phi_1 + [\delta_1, \lambda] - \frac{1}{2}\llbracket\lambda, \lambda\rrbracket.
\end{align*}
\end{itemize}
\end{defn}

Similarly, one has the notion of a \defterm{quasi-Poisson group} $G$ where one replaces the cobracket $\delta$ by a bivector $\pi\in\Gamma(G, \wedge^2(\T_G))$, see e.g. \cite[Definition 2.1]{KM}. We denote by $\QPois(G)$ the groupoid of quasi-Poisson structures on $G$ with morphisms given by twists. Quasi-Poisson groups with $\phi=0$ are called \defterm{Poisson groups}.

\begin{remark}
We have a natural ``differentiation functor'' $\QPois(G)\rightarrow \QLie(\g)$ which is fully faithful by definition.
\label{rmk:QPoisCartesian}
\end{remark}

Given a quasi-Poisson group $(G, \pi, \phi)$, the linear part of $\pi$ at the unit defines a quasi-Lie bialgebra structure on $\g$. Conversely, if $\g$ is a quasi-Lie bialgebra, the group $G$ has at most one compatible quasi-Poisson structure.

Recall that a $G$-scheme $X$ is simply a smooth scheme $X$ together with a $G$-action $G\times X\rightarrow X$. We denote its components by $R_x\colon G\rightarrow X$ and $L_g\colon X\rightarrow X$ for $x\in X$ and $g\in G$. Let $a\colon \g\rightarrow \Gamma(X, \T_X)$ be the infinitesimal action map.

\begin{defn}
Let $(G, \pi_G, \phi)$ be a quasi-Poisson group. A \defterm{quasi-Poisson $G$-scheme} is a $G$-scheme $X$ together with a bivector $\pi_X\in\Gamma(X, \wedge^2 \T_X)$ satisfying
\begin{align}
\pi_X(gx) &= L_{g, *} \pi_X(x) + R_{x, *} \pi_G(g) \label{eq:qpoissonspace1} \\
\frac{1}{2}[\pi_X, \pi_X] + a(\phi) &= 0. \label{eq:qpoissonspace2}
\end{align}
\end{defn}

\begin{defn}
Let $G$ be a group and $X$ a $G$-scheme. Let $\QPois(G, X)$ be the following groupoid:
\begin{itemize}
\item Objects of $\QPois(G, X)$ are quasi-Poisson group structures on $G$ and bivectors $\pi_X$ on $X$ making $X$ into a quasi-Poisson $G$-scheme.

\item Morphisms of $\QPois(G, X)$ from $(\pi_G, \phi, \pi_X)$ to $(\pi_G', \phi', \pi_X')$ are given by elements $\lambda\in\wedge^2(\g)$ defining a morphism $(\pi_G, \phi)\rightarrow (\pi_G', \phi')$ in $\QPois(G)$ and satisfying
\[\pi_X' = \pi_X + a(\lambda).\]
\end{itemize}
\end{defn}

\subsection{Shifted Poisson structures on \texorpdfstring{$\B G$}{BG}}

In this section we classify shifted Poisson structures on the classifying stack $\B G$. The following statement also appears in \cite[Section 3.6.2]{CPTVV}:
\begin{prop}
The space of $n$-shifted Poisson structures on $\B G$ is trivial if $n > 2$. Moreover, we can identify the space $\Pois(\B G, 2)$ with the set $\Sym^2(\g)^G$.
\label{prop:2shiftedBG}
\end{prop}
\begin{proof}
The algebra of $n$-shifted polyvectors is
\[\Pol(\B G, n)\cong \C^\bullet(G, \Sym(\g[-n]))\]
as a graded complex. In particular, we obtain a graded homotopy $\bP_{n+2}$ structure on $\C^\bullet(G, \Sym(\g[-n]))$. Elements of weight at least 2 have cohomological degree at least $2n$. But Maurer--Cartan elements in a $\bP_{n+2}$-algebra are in degree $n + 2$. Therefore, for $n>2$ the space of Maurer--Cartan elements in $\Pol(\B G, n)^{\geq 2}$ is contractible.

For $n=2$ Maurer--Cartan elements are degree $4$ elements in $\Pol(\B G, 2)^{\geq 2}$, i.e. elements $c\in\Sym^2(\g)$, satisfying
\[\d c = 0, \qquad [c, c] = 0, \qquad [c, c, c] = 0, \qquad\ldots\]

The bracket $[-, \ldots, -]_m$ in a graded homotopy $\bP_4$-algebra has weight $1-m$ and degree $5 - 4m$. Therefore, $[c, \ldots, c]_m$ is an element of weight $m+1$ and degree $5$ which is automatically zero if $m>1$. Therefore, the Maurer--Cartan equation reduces to $\d c = 0$, i.e. $c\in\Sym^2(\g)^G$.
\end{proof}

For example, identifying nondegenerate $n$-shifted Poisson structures with $n$-shifted symplectic structures by \cite[Theorem 3.2.4]{CPTVV} we get the following statement which was proved in \cite{PTVV} for reductive groups.
\begin{prop}
Let $G$ be an algebraic group. The space $\mathrm{Symp}(\B G, 2)$ of $2$-shifted symplectic structures on $\B G$ is equivalent to the subset of $\Sym^2(\g^*)^G$ consisting of non-degenerate quadratic forms.
\end{prop}

The space of 1-shifted Poisson structures is more interesting. We begin with the case of Lie algebras.

\begin{thm}
Let $\g$ be a Lie algebra of an algebraic group $G$. Then the space of 1-shifted Poisson structures $\Pois(\B\g, 1)$ is equivalent to the groupoid $\QLie(\g)$ of quasi-Lie bialgebra structures on $\g$.
\label{thm:1shiftedBg}
\end{thm}
\begin{proof}
We can identify
\[\Pois(\B\g, 1)\cong \MC(\Pol(\B\g, 1)^{\geq 2}[2]).\]

Since $\Pol(\B\g, 1)^{\geq 2}\cong \C^\bullet(\g, \Sym^{\geq 2}(\g[-1]))$, the dg Lie algebra $\Pol(\B\g, 1)^{\geq 2}[2]$ is concentrated in non-negative degrees and hence by \cref{prop:MCDeligne} we can identify its $\infty$-groupoid of Maurer--Cartan elements with the Deligne groupoid.

A Maurer--Cartan element $\alpha$ in $\Pol(\B\g, 1)^{\geq 2}[2]$ is given by a pair of elements $\delta\colon \g\rightarrow \wedge^2(\g)$ in weight 2 and $\phi\in\wedge^3(\g)$ in weight 3 which satisfy equations \eqref{eq:qlie1} - \eqref{eq:qlie3} defining the quasi-Lie bialgebra structure on $\g$.

A 1-morphism in the Deligne groupoid of $\Pol(\B\g, 1)^{\geq 2}[2]$ is given by a $t$-parameter family of quasi-Lie bialgebra structures $\delta(t)$ and $\phi(t)$ and an element $\lambda\in\wedge^2(\g)$ which satisfy the equations
\begin{align*}
\frac{d\delta(t)}{dt} - \dCE \lambda &= 0 \\
\frac{d\phi(t)}{dt} - [\delta(t), \lambda] &= 0.
\end{align*}

The first equation implies that
\[\delta(t) = \delta_0 + t \dCE\lambda.\]
Substituting it in the second equation and integrating, we obtain
\[\phi(t) = \phi_0 + t[\delta_0, \lambda] + \frac{t^2}{2} [\dCE\lambda, \lambda] = 0.\]

In this way we see that the Deligne groupoid is isomorphic to the groupoid $\QLie(\g)$ of quasi-Lie bialgebra structures on $\g$.
\end{proof}

We also have a global version of the previous statement.

\begin{thm}
Let $G$ be an algebraic group. Then the space of 1-shifted Poisson structures $\Pois(\B G, 1)$ is equivalent to the groupoid $\QPois(G)$ of quasi-Poisson structures on $G$.
\label{thm:1shiftedBG}
\end{thm}
\begin{proof}
By \cref{thm:poissonpolyvectors}
\[\Pois(\B G, 1)\cong \MC(\Pol(\B G, 1)^{\geq 2}[2]).\]
Moreover, we can identify the latter space with
\[\MC((\Pol(\B\g, 1)^{\geq 2})^{G_{\dR}}[2]).\]
We can present $G_{\dR}$-invariants on a complex $V$ as a totalization of the cosimplicial object
\[\cosimp{V}{V\otimes \Omega^\bullet(G)}.\]
Therefore, by \cref{prop:MClimits} we can identify $\MC(\Pol(\B G, 1)^{\geq 2}[2])$ with the totalization of
\[\cosimp{\MC(\Pol(\B\g, 1)^{\geq 2}[2])}{\MC(\Omega^\bullet(G)\otimes\Pol(\B\g, 1)^{\geq 2}[2])}\]

The latter is a cosimplicial diagram of groupoids and we can use \cref{lm:groupoidTot} to compute its totalization $\cG$. Let us begin by describing the groupoid $\cG^1=\MC(\Omega^\bullet(G)\otimes \Pol(\B\g, 1)^{\geq 2}[2])$. Degree $1$ elements in the corresponding dg Lie algebra are given by elements
\begin{align*}
\delta&\in\cO(G)\otimes \g^*\otimes \wedge^2(\g) \\
\phi&\in\cO(G)\otimes \wedge^3(\g) \\
P&\in\Omega^1(G)\otimes \wedge^2(\g).
\end{align*}

The Maurer--Cartan equation for $\delta+\phi+P$ boils down to the equations expressing the fact that the pair $\delta(g),\phi(g)$ defines a Lie bialgebra structure on $\g$ for any $g\in G$ and the following equations:
\begin{align*}
\ddr P &= 0 \\
\ddr \delta + \dCE P &= 0 \\
\ddr \phi + [P, \delta] &= 0.
\end{align*}

Morphisms from $(\delta,\phi, P)$ to $(\delta', \phi', P')$ in $\cG^1$ are given by elements $\pi\in\cO(G)\otimes \wedge^2(\g)$ satisfying the equations
\begin{align*}
\delta' &= \delta + \dCE\pi \\
\phi' &= \phi + [\delta, \pi] - \frac{1}{2}\llbracket\pi, \pi\rrbracket \\
P' &= P + \ddr \pi.
\end{align*}

Thus, by \cref{lm:groupoidTot} an object of $\cG$ consists of a quasi-Lie bialgebra structure $(\delta, \phi)$ on $\g$ together with a bivector $\pi\in \cO(G)\otimes \wedge^2(\g)$ satisfying the equations
\begin{align}
\pi(e) &= 0 \label{eq:poissonlie1} \\
\pi(xy) &= \pi(x) + \Ad_x\pi(y) \label{eq:poissonlie2} \\
\Ad_g(\delta) &= \delta + \dCE \pi \label{eq:poissonlie3} \\
\Ad_g(\phi) &= \phi + [\delta, \pi] - \frac{1}{2}\llbracket\pi, \pi\rrbracket \label{eq:poissonlie4} \\
\Ad_g(\delta) &= \ddr \pi. \label{eq:poissonlie5}
\end{align}

Equation \eqref{eq:poissonlie1} follows from \eqref{eq:poissonlie2} which expresses the fact that $\pi\in\Gamma(G, \wedge^2\T_G)$ is multiplicative. Denote by $\widetilde{\delta}\colon \g\rightarrow \wedge^2(\g)$ the linear part of $\pi$ at the unit. Then the infinitesimal version of \eqref{eq:poissonlie2} is
\[\ddr \pi = \widetilde{\delta} + \dCE \pi.\]
Combining this with equations \eqref{eq:poissonlie3} and \eqref{eq:poissonlie5} we see that $\delta=\widetilde{\delta}$. Therefore, equation \eqref{eq:poissonlie4} becomes
\[\Ad_g(\phi) - \phi = [\ddr \pi, \pi] +\frac{1}{2}\llbracket\pi, \pi\rrbracket.\]
If we denote by $[\pi, \pi]_{\mathrm{Sch}}$ the Schouten bracket of the bivector of $\pi$, then we can rewrite the above equation as
\[\Ad_g(\phi) - \phi = -\frac{1}{2} [\pi, \pi]_{\mathrm{Sch}}.\]
This shows that an object of $\cG$ is a quasi-Poisson structure $(\pi, \phi)$ on $G$ and $(\delta, \phi)$ is the induced Lie bialgebra structure on $\g$.

Morphisms $(\delta, \phi, \pi)\rightarrow (\delta', \phi', \pi')$ in $\cG$ are given by morphisms in $\cG^0$, i.e. by twists $\lambda\in\wedge^2(\g)$ of quasi-Lie bialgebra structures $(\delta, \phi)\rightarrow (\delta', \phi')$. The compatibility of the twist $\lambda$ with the bivectors is the equation
\[\pi + \lambda = \Ad_g(\lambda) + \pi'\]
which shows that $\lambda$ is also a twist of the quasi-Poisson structure $(\pi, \phi)$ into $(\pi', \phi')$. In this way we have identified objects of $\cG$ with quasi-Poisson structures on $\cG$ and morphisms with twists of those, i.e. $\cG\cong \QPois(G)$.
\end{proof}

\subsection{Coisotropic structures}

In this section we study coisotropic structures in $\B G$. Denote by $\QPois_H(G)\subset \QPois(G)$ the following subgroupoid:
\begin{itemize}
\item Objects of $\QPois_H(G)$ are objects $(\pi, \phi)$ of $\QPois(G)$ such that the image of $\pi$ under
\[\cO(G)\otimes \wedge^2(\g)\rightarrow \cO(H)\otimes \wedge^2(\g/\h)\]
is zero and $\phi$ is in the kernel of $\wedge^3(\g)\rightarrow \wedge^3(\g/\h)$.

\item Morphisms in $\QPois_H(G)$ are morphisms $(\pi, \phi)\rightarrow (\pi', \phi')$ in $\QPois(G)$ given by $\lambda$ in the kernel of $\wedge^2(\g)\rightarrow \wedge^2(\g/\h)$.
\end{itemize}

Note that $(\pi, \phi=0)$ is an object of $\QPois_H(G)$ iff $\pi$ makes $G$ into a Poisson-Lie group and $H\subset G$ into a coisotropic subgroup.

\begin{prop}
Let $H\subset G$ be a closed subgroup and $f\colon \B H\rightarrow \B G$ the induced map on classifying stacks.

The space of 2-shifted coisotropic structures $\Cois(f, 2)$ is equivalent to the set
\[\ker(\Sym^2(\g)^G\rightarrow \Sym^2(\g/\h)^H).\]

The space of 1-shifted coisotropic structures $\Cois(f, 1)$ is equivalent to the groupoid $\QPois_H(G)$.
\label{prop:coisotropicsubgroup}
\end{prop}
\begin{proof}
Both statements immediately follow from the description of relative polyvectors $\Pol(f, n)$ as a kernel of \[\C^\bullet(G, \Sym(\g[-n]))\rightarrow \C^\bullet(H, \Sym(\g/\h[-n])).\]
\end{proof}

\begin{cor}
Let $f\colon \pt\rightarrow \B G$ be the inclusion of the basepoint. Then the space of 1-shifted Poisson structures $\Pois(f, 1)$ is equivalent to the set of Poisson-Lie structures on $G$.
\label{cor:poissongroups}
\end{cor}
\begin{proof}
By \cref{thm:graphcoisotropic} we have an equivalence of spaces $\Pois(f, 1)\cong \Cois(f, 1)$ and by \cref{prop:coisotropicsubgroup} we have an equivalence of spaces $\Cois(f, 1)\cong \QPois_{\{e\}}(G)$. But it is immediate that the latter groupoid is in fact a set whose objects have $\phi=0$.
\end{proof}

Recall that we have a diagram of spaces
\[
\xymatrix{
& \Cois(f, 2) \ar[dl] \ar[dr] & \\
\Pois(\B H, 1) && \Pois(\B G, 2).
}
\]
Thus, given a Casimir element $c\in\Sym^2(\g)^G$ vanishing on $\Sym^2(\g/\h)^H$ we obtain a quasi-Poisson structure on $H$, let us now describe it explicitly. For simplicity we will describe the induced quasi-Lie bialgebra structure on $\h$.

Pick a splitting $\g\cong \h\oplus \g/\h$ as vector spaces. Let $\{e_i\}$ be a basis of $\h$ and $\{\te_i\}$ a basis of $\g/\h$. Let $\{e^i\}$ and $\{\te^i\}$ be the dual bases. We denote the structure constants as follows:
\begin{align*}
[e_i, e_j] &= f_{ij}^k e_k \\
[e_i, \te_j] &= A_{ij}^k e_k + B_{ij}^k \te_k \\
[\te_i, \te_j] &= C_{ij}^k e_k + D_{ij}^k \te_k
\end{align*}

Let us also split $c = P + Q$, where $P\in\Sym^2(\h)$ and $Q\in \h\otimes \g/\h$. We denote its components by $P^{ij}$ and $Q^{ij}$.

The cobracket $\delta\in \h^*\otimes \wedge^2(\h)$ has components
\begin{equation}
\delta^{ij}_k = \frac{1}{2}(A^j_{ka} Q^{ia} - A^i_{ka} Q^{ja})
\label{eq:forgetqlie1}
\end{equation}
and the associator $\phi\in\wedge^3(\h)$ has components
\begin{equation}
\phi^{ijk} = \frac{1}{8} f^i_{ab}P^{aj} P^{bk} + \frac{1}{4} Q^{ia} (\gamma^k_{ab} Q^{jb} - \gamma^j_{ab} Q^{kb}) + \frac{1}{8}P^{ia} (\alpha^k_{ab} Q^{jb} - \alpha^j_{ab} Q^{kb}).
\label{eq:forgetqlie2}
\end{equation}

A tedious computation shows that the above formulas indeed define a quasi-Lie bialgebra structure on $\h$. One can also see that if $c$ is nondegenerate and both $\h$ and $\g/\h$ are Lagrangian in $\g$, then the formulas reduce to those of \cite[Section 2]{Dr1} and \cite[Section 2.1]{AKS}.

\begin{remark}
Consider an exact sequence of $\h$-representations
\[0\longrightarrow \h\longrightarrow \g\longrightarrow \g/\h\longrightarrow 0.\]
The connecting homomorphism gives rise to a morphism $\g/\h\rightarrow \h[1]$ in the derived category of $\h$-representations. Since $\h\subset \g$ is coisotropic, the Casimir element induces a morphism of $\h$-representations $\h^*\rightarrow \g/\h$. Combining the two, we get a morphism $\h^*\rightarrow \h[1]$, i.e. an element of $\H^1(\h, \h\otimes \h)$. The formula \eqref{eq:forgetqlie1} gives an explicit representative for this cohomology class given a splitting $\g/\h\subset \g$.
\end{remark}

\begin{prop}
Let $\hat{f}\colon \B\h\rightarrow \B\g$. Under the forgetful map
\[\Cois(\hat{f}, 2)\rightarrow \Pois(\B\h, 1)\]
the induced quasi-Lie bialgebra structure on $\h$ is given by formulas \eqref{eq:forgetqlie1} and \eqref{eq:forgetqlie2}.
\label{prop:coisotropicqlie}
\end{prop}
\begin{proof}
Pick $c\in\ker(\Sym^2(\g)^G\rightarrow \Sym^2(\g/\h)^H)$ as above. We are going to construct a $\bP_{[3, 2]}$-algebra structure on the pair $(\C^\bullet(\g, k), \C^\bullet(\h, k))$ such that the $\bP_2$-structure on $\C^\bullet(\h, k)$ is induced by the above quasi-Lie bialgebra structure on $\h$. That is, we have to define a morphism of $\bP_3$-algebras
\[F\colon \C^\bullet(\g, k)\longrightarrow\rZ(\C^\bullet(\h, k))\cong \C^\bullet(\h, \Sym(\h[-1])).\]

The Poisson bracket on $\C^\bullet(\g, k)$ is given by
\[\{e^i, e^j\} = P^{ij},\qquad \{e^i, \te^j\} = Q^{ij}\]
and the Poisson bracket on $\C^\bullet(\h, \Sym(\h[-1]))$ is given by the canonical pairing of $\h^*$ and $\h$.

The differentials on $\C^\bullet(\g, k)$ are
\[\d e^i = A^i_{jk} e^j\te^k + \frac{1}{2}C^i_{jk}\te^j\te^k + \frac{1}{2}f^i_{jk} e^j e^k\]
and
\[\d \te^i = B^i_{jk} e^j\te^k + \frac{1}{2}D^i_{jk} \te^j \te^k.\]

The differentials on $\C^\bullet(\h, \Sym(\h[-1]))$ are
\[\d e^i = \frac{1}{2} f^i_{jk} e^j e^k + \phi^{ijk} e_je_k - \delta^{ij}_k e^ke_j\]
and
\[\d e_i = f^k_{ij} e^j e_k + \frac{1}{2} \delta_i^{jk} e_j e_k\]

$G$-invariance of the Casimir $c$ implies the following equations:
\begin{align}
0 &= A^i_{jk} P^{ja} + C^i_{jk} Q^{aj} + A^a_{jk} P^{ji} + C^a_{jk} Q^{ij} \label{eq:casimirinv1} \\
0 &= A^i_{jk} Q^{ak} - f^i_{kj} P^{ka} + A^a_{jk} Q^{ik} - f^a_{kj} P^{ki} \label{eq:casimirinv2} \\
0 &= -A^i_{jk} Q^{ja} - B^a_{jk}P^{ij} - D^a_{jk} Q^{ij} \label{eq:casimirinv3} \\
0 &= -f^i_{kj} Q^{ka} + B^a_{jk} Q^{ik} \label{eq:casimirinv4} \\
0 &= B^i_{jk} Q^{ja} + B^a_{jk} Q^{ji}. \label{eq:casimirinv5}
\end{align}

Since $\C^\bullet(\g, k)$ is generated as a commutative algebra by $\g^*$, it is enough to define $F$ on the generators: we define
\[F(e^i) = e^i + \frac{1}{2}P^{ij} e_j,\qquad F(\te^i) = Q^{ji} e_j.\]
It is easy to see that thus defined $F$ is compatible with the Poisson brackets. We are now going to show that it is also compatible with differentials.

Consider $\te^i\in(\g/\h)^*$. Then
\[\d \te^i = B^i_{jk} e^j\te^k + \frac{1}{2}D^i_{jk} \te^j \te^k\]
and hence
\[F(\d \te^i) = B^i_{jk} \left(e^j + \frac{1}{2} P^{ja} e_a\right) Q^{bk} e_b + \frac{1}{2} D^i_{jk} Q^{aj} Q^{bk} e_a e_b.\]

On the other hand,
\[\d F(\te^a) = Q^{ia} f^k_{ij} e^j e_k + \frac{1}{2} Q^{ia} \delta_i^{jk} e_j e_k.\]

Therefore,
\[
Q^{ci} \delta_c^{ab} = \frac{1}{2}(B^i_{jk} P^{ja} Q^{bk} + D^i_{jk} Q^{aj} Q^{bk}) - (a\leftrightarrow b).
\]
which holds by \eqref{eq:casimirinv3}.

Now consider $e^i\in\h^*$. Then
\[\d e^i = A^i_{jk} e^j\te^k + \frac{1}{2}C^i_{jk}\te^j\te^k + \frac{1}{2}f^i_{jk} e^j e^k\]
and hence
\[F(\d e^i) = A^i_{jk} \left(e^j + \frac{1}{2}P^{ja}e_a\right) Q^{bk}e_b + \frac{1}{2}C^i_{jk} Q^{aj}Q^{bk}e_ae_b + \frac{1}{2} f^i_{jk} \left(e^j + \frac{1}{2} P^{ja}e_a\right)\left(e^k + \frac{1}{2} P^{kb}e_b\right).\]

On the other hand,
\[\d F(e^i) = \frac{1}{2} f^i_{jk} e^j e^k + \phi^{iab} e_ae_b - \delta^{ib}_j e^je_b + \frac{1}{2} P^{ib}\left(f^k_{bj} e^j e_k + \frac{1}{2} \delta_b^{jk} e_j e_k\right).\]

Therefore,
\begin{align*}
P^{ik} f^b_{kj} - 2 \delta^{ib}_j &= 2A^i_{jk} Q^{bk} + f^i_{jk} P^{kb} \\
\phi^{iab} + \frac{1}{4} P^{ic} \delta^{ab}_c &= \frac{1}{4}A^i_{jk} P^{ja} Q^{bk} - \frac{1}{4} A^i_{jk} P^{jb} Q^{ak} + \frac{1}{2}C^i_{jk} Q^{aj} Q^{bk} + \frac{1}{8} f^i_{jk} P^{ja} P^{kb}.
\end{align*}
These can be checked using \eqref{eq:casimirinv1} and \eqref{eq:casimirinv2}.
\end{proof}

Let us recall that one can identify $G$-spaces with spaces over $\B G$: given a $G$-space $X$ we get $Y=[X/G]\rightarrow \B G$; conversely, given $Y\rightarrow \B G$ we form the $G$-space $X = \pt\times_{\B G} Y$. The following statements extend this analogy to the Poisson setting.

\begin{prop}
Let $X$ be a $G$-scheme and $f\colon [X/G]\rightarrow \B G$ the projection map. Then we have an equivalence of groupoids
\[\Cois(f, 1)\cong \QPois(G, X).\]
\label{prop:poissonspaces}
\end{prop}
\begin{proof}
We have a fiber sequence of Lie algebras
\[
\xymatrix{
\Pol([X/G]/\B G, 0)[1] \ar[r] & \Pol(f, 1)[2] \ar[r] & \Pol(\B G, 1)[2]
}
\]
We have quasi-isomorphisms
\[\Pol(\B G, 1)\cong \C^\bullet(G, \Sym(\g[-1])),\qquad \Pol([X/G]/\B G, 0)\cong \C^\bullet(G, \Pol(X, 0)).\]

The connecting homomorphism
\[a\colon \Pol(\B G, 1)\longrightarrow \Pol([X/G]/\B G, 0)\]
is then induced by the action map $\g\rightarrow \Gamma(X, \T_X)$.

The Lie algebra $\Pol(f, 1)^{\geq 2}[2]$ is concentrated in non-negative degrees, so by \cref{prop:MCDeligne} we can identify its Maurer--Cartan space with the Deligne groupoid. The degree 3 elements in $\Pol(f, 1)^{\geq 2}$ are elements $(\pi_G, \phi)$ in $\Pol(\B G, 1)$ and $\pi_X\in\Gamma(X, \wedge^2 \T_X)$. The Maurer--Cartan equation reduces to the Maurer--Cartan equations in $\Pol(\B G, 1)$ which imply by \cref{thm:1shiftedBG} that $(\pi_G, \phi)$ define a quasi-Poisson structure on $G$ and the equations
\begin{align*}
\d \pi_X + a(\pi_G) &= 0 \\
\frac{1}{2}[\pi_X, \pi_X] + [\pi_X, \pi_G] + a(\phi) &= 0.
\end{align*}
The other brackets vanish due to weight reasons. But the morphism
\[\Pol(f, 1)[2]\longrightarrow \Pol([X/G], 0)[1]\]
is a morphism of Lie algebras which shows that $[\pi_X, \pi_G] = 0$. Therefore, the second equation is equivalent to
\[\frac{1}{2}[\pi_X, \pi_X] + a(\phi) = 0,\]
and therefore we recover equations \eqref{eq:qpoissonspace1} and \eqref{eq:qpoissonspace2} in the definition of quasi-Poisson $G$-spaces.

Morphisms $(\pi_X, \pi_G, \phi)\rightarrow (\pi_X', \pi_G', \phi')$ are given by degree $2$ elements in $\Pol(f, 1)^{\geq 2}$ which are elements $\lambda\in\wedge^2(\g)$. These elements, firstly, define a morphism $(\pi_G, \phi)\rightarrow (\pi_G', \phi')$ in $\QPois(G)$ by \cref{thm:1shiftedBG}  and, secondly, satisfy the equation
\[\pi_G' - \pi_G = a(\lambda).\]
In this way we have identified the Maurer--Cartan space of $\Pol(f, 1)^{\geq 2}[2]$ with $\QPois(G, X)$.
\end{proof}

Combining \cref{prop:coisotropicsubgroup} and \cref{prop:poissonspaces} we obtain the following statement:
\begin{cor}
We have an equivalence of groupoids
\[\QPois(G, G/H)\cong \QPois_H(G).\]
\end{cor}

In other words, given a coisotropic subgroup $H\subset G$ of a quasi-Poisson group $G$ we get a natural quasi-Poisson structure on the homogeneous $G$-scheme $[G/H]$.

\subsection{Forgetting the shift}
\label{sect:forgetshift}

In this section we relate $n$-shifted Poisson structures on classifying spaces to $(n-1)$-shifted Poisson structures. We begin with the case $n=2$. If $c\in\Sym^2(\g)^G$, recall the standard notation
\begin{align*}
c_{12} &= c\otimes 1\in \U(\g)^{\otimes 3}\\
c_{23} &= 1\otimes c\in \U(\g)^{\otimes 3}.
\end{align*}

\begin{prop}
The forgetful map
\[\Pois(\B G, 2)\longrightarrow \Pois(\B G, 1)\]
sends a Casimir element $c\in\Sym^2(\g)^G$ to the quasi-Poisson structure with $\pi=0$ and
\[\phi = -\frac{1}{6}[c_{12}, c_{23}].\]

In particular, if $G$ is semisimple, the forgetful map
\[\Pois(\B G, 2)\longrightarrow \Pois(\B G, 1)\]
is an equivalence.
\label{prop:forget21BG}
\end{prop}
\begin{proof}
The first claim follows from \cref{prop:coisotropicqlie} by explicitly computing \eqref{eq:forgetqlie2} in this case.

Since $G$ is semisimple, we obtain a commutative diagram
\[
\xymatrix{
\Pois(\B G, 2) \ar[r] \ar^{\sim}[d]& \Pois(\B G, 1) \ar^{\sim}[d] \\
\Sym^2(\g)^G \ar[r] & \wedge^3(\g)^G
}
\]
But it is well-known that the map $\Sym^2(\g)^G\rightarrow \wedge^3(\g)^G$ is an isomorphism, see \cite[Th\'{e}or\`{e}me 11.1]{Kos}.
\end{proof}

If $\g$ is a Lie bialgebra, the space $\g$ carries the Kirillov--Kostant--Souriau Poisson structure induced from the Lie cobracket. Recall the following notion (see e.g. \cite{EEM}):
\begin{defn}
A Poisson-Lie group $G$ is \defterm{formally linearizable} at the unit if we have an isomorphism of formal Poisson schemes
\[\widehat{\g}_{\{0\}}\rightarrow \widehat{G}_{\{e\}}\]
inducing identity on tangent spaces at the origin.
\end{defn}

Let $G^*$ be an algebraic group whose Lie algebra is $\g^*$.

\begin{prop}
Let $G$ be a Poisson-Lie group. The image of the Poisson structure under
\[\Pois(\pt\rightarrow \B\g, 1)\longrightarrow \Pois(\pt\rightarrow \B\g, 0)\]
is zero iff $G^*$ is formally linearizable.
\label{prop:forget10BG}
\end{prop}
\begin{proof}
The forgetful map on augmented shifted Poisson algebras can be obtained as a composite
\[\balg^{\aug}_{\bP_{n+1}}\stackrel{\sim}\longrightarrow\balg(\balg^{\aug}_{\bP_n})\longrightarrow \balg^{\aug}_{\bP_n}.\]

Recall the construction of the first functor from \cite{Sa3}. Using Koszul duality we can identify
\[\balg^{\aug}_{\bP_n}\cong \alg_{\coP_n}[\kos^{-1}],\]
where $\kos$ is a certain class of weak equivalences of coaugmented $\bP_n$-coalgebras. Then the image of the augmented $\bP_2$-algebra $\C^\bullet(\g, k)$ is given by the associative algebra object in Poisson coalgebras $\U(\g^*)$, whose Poisson cobracket is induced from the Lie algebra structure on $\g^*$.

We have a commutative diagram
\[
\xymatrix{
\balg^{\aug}_{\bP_n} \ar[d] \ar^-{\sim}[r] & \alg_{\coP_n}[\kos^{-1}] \ar[d] \\
\balg^{\aug}_{\Comm} \ar^-{\sim}[r] & \alg_{\coLie}[\kos^{-1}]
}
\]
where the horizontal equivalences are given by Koszul duality, the vertical functor on the left is the obvious forgetful functor and the vertical functor on the right is given by taking the primitive elements in a coaugmented $\bP_n$-coalgebra.

The Poisson coalgebra $\U(\g^*)$ has primitives given by the Lie coalgebra $\g^*$. Thus, the Koszul dual of the coaugmented Poisson coalgebra $\U(\g^*)$ is a certain augmented homotopy $\bP_1$-algebra $A$ lifting the standard augmented cdga structure on $\C^\bullet(\g, k)$.

Now consider $\C^\bullet(\g, k)$ with the zero Poisson bracket. Trivialization of the 0-shifted Poisson structure on $\B G$ is a weak equivalence of augmented $\bP_1$-algebras
\[A\cong \C^\bullet(\g, k)\]
being the identity on the underlying augmented commutative algebras.

Passing to the Koszul dual side, it corresponds to a weak equivalence of coaugmented $\bP_1$-coalgebras
\[\U(\g^*)\cong \Sym(\g^*)\]
inducing the identity on the primitives. Since both $\bP_1$-coalgebras are concentrated in degree zero, this boils down to an \emph{isomorphism} of coaugmented $\bP_1$-coalgebras which is the identity on $\g^*$. But $\U(\g^*)$ is the Poisson coalgebra of distributions on $\widehat{G^*}_{\{e\}}$ and $\Sym(\g^*)$ is the Poisson coalgebra of distributions on $\widehat{\g^*}_{\{0\}}$.
\end{proof}

\section{Poisson groupoids}

Let $\cG\rightrightarrows X$ be a groupoid with associated Lie algebroid $\cL$ on $X$. Recall that $\cG$ is called a \defterm{Poisson groupoid} if we are given a multiplicative Poisson structure $\pi$ on $\cG$. Similarly, it is called a \defterm{quasi-Poisson groupoid} if we are given a multiplicative bivector $\pi$ on $\cG$ and a trivector $\phi\in \Gamma(X, \wedge^3 \cL)$ satisfying some equations generalizing those of quasi-Poisson groups. Infinitesimal versions of quasi-Poisson groupoids are known as \defterm{quasi-Lie bialgebroids}. The goal of this section is to relate these notions to shifted Poisson structures.

\subsection{Quasi-Lie bialgebroids}

Suppose $X$ is a smooth affine scheme and $\cG\rightrightarrows X$ is a groupoid with associated Lie algebroid $\cL$. Our assumption in this section will be that $\cG$ is also an affine scheme with the source map being smooth. In particular, $\cG$ is smooth itself. Let us recall the definition of quasi-Lie bialgebroids from \cite{IPLGX}.

\begin{defn}
An \defterm{almost $p$-differential on $\cL$} is a pair of $k$-linear maps
\[\delta\colon \cO(X)\rightarrow \Gamma(X, \wedge^{p-1} \cL),\qquad \delta\colon \Gamma(X, \cL)\rightarrow \Gamma(X, \wedge^p \cL)\]
satisfying
\begin{align*}
\delta(fg) &= \delta(f)g + f\delta(g) & f,g\in\cO(X) \\
\delta(fv) &= \delta(f)\wedge v + f\delta(v) & f\in\cO(X),v\in\Gamma(X, \cL).
\end{align*}
\end{defn}

The algebra $\Gamma(X, \Sym(\cL[-1]))$ has a natural Gerstenhaber bracket that we denote by $\llbracket -, -\rrbracket$.

\begin{defn}
A \defterm{$p$-differential on $\cL$} is an almost $p$-differential satisfying
\begin{align*}
\delta \llbracket v, f\rrbracket &= \llbracket \delta v, f\rrbracket + \llbracket v, \delta f\rrbracket & f\in\cO(X), v\in\Gamma(X, \cL) \\
\delta \llbracket v, w\rrbracket &= \llbracket \delta v, w\rrbracket + \llbracket v, \delta w\rrbracket & v,w\in\Gamma(X, \cL)
\end{align*}
\end{defn}

Given $\lambda\in\Gamma(X, \wedge^p \cL)$ we have a $p$-differential $\ad(\lambda) = \llbracket \lambda, -\rrbracket$.

\begin{defn}
The \defterm{groupoid $\MultBivec(\cL)$ of multiplicative bivector fields on $\cL$} is defined as follows:
\begin{itemize}
\item Its objects are $2$-differentials $\delta$ on $\cL$.

\item Morphisms from $\delta$ to $\delta'$ are given by elements $\lambda\in\Gamma(X, \wedge^2\cL)$ such that
\[\delta' = \delta + \ad(\lambda).\]
\end{itemize}
\end{defn}

\begin{defn}
The \defterm{groupoid $\QLieAlgd(\cL)$ of quasi-Lie bialgebroid structures on $\cL$} is defined as follows:
\begin{itemize}
\item Its objects are pairs $(\delta, \Omega)$, where $\delta$ is a 2-differential on $\cL$ and $\Omega\in\Gamma(X, \wedge^3 \cL)$ which satisfy the following equations:
\begin{align*}
\frac{1}{2}[\delta, \delta] &= \ad(\Omega) \\
\delta\Omega &= 0.
\end{align*}

\item Morphisms are given by elements $\lambda\in\Gamma(X, \wedge^2\cL)$ which identify two quasi-Lie bialgebroid structures by a twist (see \cite[Section 4.4]{IPLGX}).
\end{itemize}
\end{defn}

By construction there is an obvious forgetful map $\QLieAlgd(\cL)\rightarrow \MultBivec(\cL)$.

\begin{remark}
The definition of quasi-Lie bialgebroids we use is given in \cite[Definition 4.6]{IPLGX}. They were originally introduced by Roytenberg in \cite{Ro}, but note that what Roytenberg calls a quasi-Lie bialgebroid is dual to that of \cite{IPLGX}.
\end{remark}

Given a groupoid $\cG$ over $X$ we have a notion of a left $\cG$-space which is a space $Y\rightarrow X$ equipped with an associative action map $\cG\times_X Y\rightarrow Y$. Note that since the morphism $\cG\rightarrow X$ is smooth, we may consider the underived fiber product. For example, the groupoid $\cG$ itself is a left (resp. right) $\cG$-space if we consider $\cG\rightarrow X$ to be the source (resp. target) map.

We consider the following associated spaces:
\begin{itemize}
\item The groupoid $\widehat{\cG}\rightrightarrows X$ is the formal completion of $\cG$ along the unit section.

\item We denote $[X/\cL] = [X/\widehat{\cG}]$.

\item Observe that $\cG$ equipped with the target map is a right $\widehat{\cG}$-space. Let $[\cG/\cL] = [\cG / \widehat{\cG}]$. The source map defines a projection $[\cG/\cL]\rightarrow X$ and the target map defines a projection $[\cG/\cL]\rightarrow [X/\cL]$. The space $[\cG/\cL]\rightarrow X$ is a left $\cG$-space. Moreover, we have an equivalence
\[\left[\cG\backslash [\cG/\cL]\right]\cong [X/\cL].\]
\end{itemize}

Let $\cM$ be an $\cL$-module, i.e. an $\cO(X)$-module equipped with a compatible action of $\Gamma(X, \cL)$. We have the Chevalley--Eilenberg graded mixed cdga
\[\C^\epsilon(\cL, \cM) = \Sym_{\cO(X)}(\cL^*[-1])\otimes_{\cO(X)} \cM\]
with the zero differential and the mixed structure generalizing the mixed structure on the Chevalley--Eilenberg complex $\C^\epsilon(\g, \cM)$ of a Lie algebra. Let
\[\C^\epsilon(\cL)=\C^\epsilon(\cL, \cO_X).\]

Denote by $\Omega^\epsilon(\cG/t)$ the de Rham complex of forms on $\cG$ relative to the target map $t\colon \cG\rightarrow X$. Identifying $\cL$ with right-invariant vector fields we get an isomorphism of graded cdgas
\[\Omega^\epsilon(\cG/t)\cong \C^\epsilon(\cL, t_*\cO(\cG)).\]

\begin{lm}
The isomorphism of graded cdgas
\[\Omega^\epsilon(\cG/t)\cong \C^\epsilon(\cL, t_*\cO(\cG))\]
is compatible with the mixed structures.
\label{lm:CEforms}
\end{lm}

Using this Lemma we can now describe $\bD([X/\cL])$.

\begin{thm}
One has an equivalence of graded mixed cdgas
\[\bD([X/\cL])\cong \C^\epsilon(\cL)\]
compatible with the actions of $[\dquot{\cL}{\cG}{\cL}]$ on both sides.
\end{thm}
\begin{proof}
The proof is similar to the proof of \cref{thm:DBg}.

Observe that $\bD(\cG)\cong \cO(\cG)$ is commutative Hopf algebroid over $\cO(X)$ in the sense of \cite[Definition A1.1.1]{Rav}. The fact that $[\cG/\cL]$ is a left $\cG$-space translates on the level of functions to a coaction map
\[\bD([\cG/\cL])\longrightarrow \cO(\cG)\otimes_{\cO(X)}\bD([\cG/\cL])\]
where we use that both $\cG$ and $X$ are reduced, so $\bD(\cG)\cong \cO(\cG)$.

Using the isomorphism
\[[X/\cL]\cong\left[\cG\backslash [\cG/\cL]\right]\]
we can identify $\bD([X/\cL])\cong \bD([\cG/\cL])^{\cG}$. By \cref{prop:deRhamfunctions} we can identify $\bD([\cG/\cL])\cong \Omega^\epsilon(\cG/t)$ and by \cref{lm:CEforms} we identify $\Omega^\epsilon(\cG/t)\cong \C^\epsilon(\cL, t_*\cO(\cG))$. We can identify $\cG$-invariants on $t_*\cO(\cG)$ with $\cO([\cG\backslash \cG])\cong \cO(X)$, so we finally get
\[\bD([X/\cL])\cong \C^\epsilon(\cL, \cO(X)).\]
\end{proof}

The following notion was introduced in \cite{AAC}.

\begin{defn}
A \defterm{representation of $\cL$ up to homotopy} is a $\C^\epsilon(\cL)$-module $\cE$ which is equivalent to $\C^\epsilon(\cL)\otimes_{\cO(X)} \cM$ as a graded $\C^\epsilon(\cL)$-module for some $\cO(X)$-module $\cM$. A representation up to homotopy is \defterm{perfect} if $\cM$ is perfect.
\end{defn}

Note that the underlying complex $\cM$ is uniquely determined to be the weight 0 part of the $\C^\epsilon(\cL)$-module $\cE$. Given a representation up to homotopy $\cE$ with underlying complex $\cM$ we denote
\[\C^\bullet(\cL, \cM) = |\cE|.\]

Representations of $\cL$ up to homotopy form a dg category $\Rep \cL$ defined as the full subcategory of $\Mod^{gr, \epsilon}_{\C^\epsilon(\cL)}$. We denote by $\Rep_{perf} \cL\subset \Rep \cL$ the full subcategory consisting of perfect complexes.

\begin{remark}
There is a fully faithful functor from the derived category of $\cL$-representations to $\Rep \cL$. We expect it to be an equivalence, but we will not use this statement in the paper.
\end{remark}

\begin{prop}
We have an equivalence of dg categories
\[\Perf([X/\cL])\cong \Rep_{perf} \cL.\]
Under this equivalence we have the following:
\begin{itemize}
\item The pullback functor $\Perf([X/\cL])\rightarrow \Perf(X)$ is the forgetful functor $\Rep_{perf}\cL\rightarrow \Perf(X)$.

\item The pushforward functor $\Perf([X/\cL])\rightarrow \Vect$ is given by $\cE\mapsto \C^\bullet(\cL, \cM)$.
\end{itemize}
\label{prop:PerfXL}
\end{prop}
\begin{proof}
The stack $[X/\cL]$ is an algebraisable affine formal derived stack, so by \cite[Theorem 2.2.2]{CPTVV} we can identify $\Perf([X/\cL])$ with the full subcategory $\Mod^{gr,\epsilon,perf}_{\C^\epsilon(\cL)}$ of graded mixed $\C^\epsilon(\cL)$-modules $\cE$ which are equivalent to $\C^\epsilon(\cL)\otimes_{\cO(X)} \cM$ as graded $\C^\epsilon(\cL)$-modules for some perfect complex $\cM$ on $X$. This proves the first claim.

The pullback $\Perf([X/\cL])\rightarrow \Perf(X)$ sends the $\C^\epsilon(\cL)$-module $\cE$ to $\cE\otimes_{\C^\epsilon(\cL)} \cO(X)$, i.e. it corresponds to taking the weight 0 part of $\cE$.

The pushforward $\Perf([X/\cL])\rightarrow \Vect$ is given by $\cE\mapsto \Hom(\cO_X, \cE)$. The latter complex is equivalent to $|\cE|\cong \C^\bullet(\cL, \cM)$.
\end{proof}

We are now going to describe $1$-shifted Poisson structures on $[X/\cL]$. As a first step, we will need a description of $\Pol([X/\cL], 1)$.

\begin{prop}
The graded dg Lie algebra of polyvectors $\Pol([X/\cL], 1)$ has the following description.

\begin{enumerate}
\item The complex $\Pol^p([X/\cL], 1)$ is concentrated in degrees $\geq p$.

\item An element of $\Pol^p([X/\cL], 1)$ of degree $p$ is an element of $\Gamma(X, \wedge^p\cL)$.

\item An element of $\Pol^p([X/\cL], 1)$ of degree $p+1$ is an almost $p$-differential on $\cL$.

\item The differential on $\Pol^p([X/\cL], 1)$ from degree $p$ to degree $p+1$ is given by $\lambda\mapsto \ad(\lambda)$.

\item The closed elements of $\Pol^p([X/\cL], 1)$ of degree $p+1$ are $p$-differentials on $\cL$.
\end{enumerate}
\end{prop}
\begin{proof}
Recall that $\Pol^p([X/\cL], 1)$ is the complex of $p$-multiderivations of $\bD([X/\cL])\cong \C^\epsilon(\cL)$. Since $\cO(X)$ is smooth and $\C^\epsilon(\cL)$ is freely generated as an $\cO(X)$-algebra, we may consider strict multiderivations. As a graded commutative algebra, $\C^\epsilon(\cL)$ is generated by $\cO(X)$ in degree 0 and $\Gamma(X, \cL^*)$ in degree 1. Therefore, $p$-multiderivations are uniquely determined by their values on these spaces. In other words, an element of $\Pol^p([X/\cL], 1)$ of degree $p+n$ is given by a collection of maps
\[\alpha^{(l)}\colon \Sym^l(\cO(X))\otimes_k \wedge^{p-l}\Gamma(X, \cL^*)\rightarrow \Gamma(X, \wedge^{n-l} \cL^*)\]
which together satisfy the Leibniz rule. Here $0\leq l\leq p$ and $l\leq n$. In particular, these are nonzero only for $n\geq 0$ which proves the first statement.

For $n=0$ only $\alpha^{(0)}$ is nonzero which gives an $\cO(X)$-linear map $\Gamma(X, \wedge^p \cL^*)\rightarrow \cO(X)$, i.e. $\alpha^{(0)}\in \Gamma(X, \wedge^p \cL)$.

For $n=1$ the nontrivial maps are
\begin{align*}
\alpha^{(0)}&\colon \wedge^p\Gamma(X, \cL^*)\rightarrow \Gamma(X, \cL^*) \\
\alpha^{(1)}&\colon \cO(X)\otimes \wedge^{p-1}\Gamma(X, \cL^*)\rightarrow \cO(X).
\end{align*}

The Leibniz rule implies that $\alpha^{(1)}\in\Gamma(X, \wedge^{p-1}\cL\otimes \T_X)$ and
\begin{equation}
\alpha^{(0)}(l_1, \dots, fl_p) = f\alpha^{(0)}(l_1, \dots, fl_p) + \iota_{\ddr f}\alpha^{(1)}(l_1, \dots, l_{p-1})\wedge l_p
\label{eq:alphaLeibniz}
\end{equation}
for $f\in\cO(X)$ and $l_1,\dots, l_p\in\Gamma(X, \cL^*)$.

Using these maps we define a $p$-differential on $\cL$ as follows. For $f\in\cO(X)$ let
\[\delta(f) = \iota_{\ddr f}\alpha^{(1)}.\]
For $s\in\Gamma(X, \cL)$ and $l_1,\dots, l_p\in \Gamma(X, \cL^*)$ let
\[\langle l_1\wedge \dots \wedge l_p, \delta(s)\rangle = \langle \alpha^{(0)}(l_1, \dots, l_p), s\rangle - \sum_{a=1}^p (-1)^{p-a}\iota_{\ddr l_a(s)} \alpha^{(1)}(l_1, \dots, \hat{l}_a, \dots, l_p).\]
From \eqref{eq:alphaLeibniz} it is easy to see that the right-hand side is $\cO(X)$-linear in $l_1,\dots, l_p$, so defines a map $\delta\colon\Gamma(X, \cL)\rightarrow \Gamma(X, \wedge^p\cL)$. This establishes an isomorphism between pairs $\{\alpha^{(0)}, \alpha^{(1)}\}$ and almost $p$-differentials.

The differential on $\Pol([X/\cL], 1)$ comes from the mixed structure on $\C^\epsilon(\cL)$, i.e. the Chevalley--Eilenberg differential. Then the last two statements are obtained by a straightforward calculation.
\end{proof}

\begin{remark}
We will not use a description of elements of $\Pol([X/\cL], 1)$ of higher degrees, but we refer to \cite{CM} for the full description of the weight 1 part of $\Pol([X/\cL], 1)$.
\end{remark}

\begin{remark}
In \cite[Section 2.1]{BCLGX} the authors define a graded Lie 2-algebra $\cV(\cG)$ of polyvector fields on a groupoid $\cG\rightrightarrows X$. One may similarly define the graded Lie 2-algebra $\cV(\cL)$ of polyvector fields on a Lie algebroid $\cL$. Let $\tau \Pol([X/\cL], 1)$ be the truncation where we truncate the weight $p$ part in degrees $\leq(p+1)$. Identifying graded Lie 2-algebras with graded dg Lie algebras whose weight $p$ part is concentrated in degree $p-2$ and $p-1$, we get an isomorphism of graded dg Lie algebras $\cV(\cL)\rightarrow (\tau \Pol([X/\cL], 1))[2]$.
\end{remark}

\begin{cor}
The space $\MC(\Pol^2([X/\cL], 1)[2])$ of Maurer--Cartan elements in the abelian dg Lie algebra $\Pol^2([X/\cL], 1)[2]$ is equivalent to the groupoid $\MultBivec(\cL)$ of multiplicative bivectors on $\cL$.
\label{cor:XLbivectors}
\end{cor}

\begin{thm}
Suppose $\cL$ is the Lie algebroid of the groupoid $\cG\rightrightarrows X$. Then the space $\Pois([X/\cL], 1)$ of $1$-shifted Poisson structures on $[X/\cL]$ is equivalent to the groupoid $\QLieAlgd(\cL)$ of quasi-Lie bialgebroid structures on $\cL$.
\label{thm:quasiliebialgebroids}
\end{thm}
\begin{proof}
The proof of this statement is completely analogous to the proof of \cref{thm:1shiftedBg}, so we omit it.
\end{proof}

\begin{cor}
With notations as before, the space $\Cois(X\rightarrow [X/\cL], 1)$ of $1$-shifted coisotropic structures on the projection $X\rightarrow [X/\cL]$ is equivalent to the set of 2-differentials $\delta$ on $\cL$ endowing it with the structure of a Lie bialgebroid.
\label{cor:liebialgebroids}
\end{cor}
\begin{proof}
As for \cref{prop:coisotropicsubgroup}, the claim follows from the description of $\Pol(X\rightarrow [X/\cL], 1)$ as the kernel of
\[\Pol([X/\cL], 1)\longrightarrow \Pol(X, 1).\]

Indeed, Maurer--Cartan elements in the kernel have $\Omega=0$ and homotopies have $\lambda=0$, i.e. all of them are given by the identity.
\end{proof}

\begin{remark}
The statement of \cref{cor:liebialgebroids} extends verbatim to non-affine schemes $X$ by using Zariski descent for shifted Poisson structures. One can similarly apply descent to compute $\Pois([X/\cL], 1)$ in the case of a non-affine $X$, but then one discovers that such ``non-affine quasi-Lie bialgebroids'' have a twisting class in $\H^1(X, \wedge^2\cL)$ which obstructs the existence of a global cobracket.
\end{remark}

\subsection{Quasi-Poisson groupoids}

Fix a smooth affine algebraic groupoid $\cG\rightrightarrows X$ as before with associated Lie algebroid $\cL$. We have an embedding $\cG\times_X\cG\hookrightarrow \cG\times \cG\times \cG$ where the first two maps are the projections are the third map is the composition.

\begin{defn}
A bivector $\Pi\in\Gamma(\cG, \wedge^2 \T_\cG)$ is \defterm{multiplicative} if the embedding $\cG\times_X\cG\hookrightarrow \cG\times \cG\times \cG$ is coisotropic with respect to $\Pi\oplus \Pi\oplus (-\Pi)$.
\end{defn}

Given $\lambda\in\Gamma(X, \wedge^2 \cL)$, we can extend it to left- and right-invariant bivectors $\lambda^L, \lambda^R\in\Gamma(\cG, \wedge^2 \T_\cG)$. The difference $\lambda^R - \lambda^L$ is then a multiplicative bivector on $\cG$.

\begin{defn}
The \defterm{groupoid $\MultBivec(\cG)$ of multiplicative bivector fields on $\cG$} is defined as follows:
\begin{itemize}
\item Its objects are multiplicative bivector fields $\Pi$ on $\cG$.

\item Morphisms from $\Pi$ to $\Pi'$ are given by $\lambda\in\Gamma(X, \wedge^2\cL)$ such that
\[\Pi' = \Pi + \lambda^R - \lambda^L.\]
\end{itemize}
\end{defn}

\begin{defn}
The \defterm{groupoid $\QPoisGpd(\cG)$ of quasi-Poisson groupoid structures on $\cG$} is defined as follows:
\begin{itemize}
\item Its objects are pairs $(\Pi, \Omega)$, where $\Pi$ is a multiplicative bivector field on $\cG$ and $\Omega\in\Gamma(X, \wedge^3\cL)$ which satisfy the following equations:
\begin{align*}
\frac{1}{2}[\Pi, \Pi] &= \Omega^R - \Omega^L \\
[\Pi, \Omega^R] &= 0
\end{align*}

\item Morphisms are given by elements $\lambda\in\Gamma(X, \wedge^2\cL)$ which identify two quasi-Poisson groupoids by a twist (see \cite[Section 4.4]{IPLGX}).
\end{itemize}
\end{defn}

By construction there is an obvious forgetful map $\QPoisGpd(\cG)\rightarrow \MultBivec(\cG)$.

In \cref{sect:BGpolyvectors} we have used the isomorphisms
\[[G\backslash G_{\dR}]\cong \B \g,\qquad [\B\g / G_{\dR}]\cong \B G\]
to compute $\Pol(\B G, n)$. In the groupoid setting these are generalized to isomorphisms
\[\left[\cG\backslash [\cG/\cL]\right]\cong [X/\cL],\qquad \mathlarger{\mathlarger{\left[\sfrac{[X/\cL]}{[\dquot{\cL}{\cG}{\cL}]}\right]}}\cong [X/\cG]\]
which can be used to compute $\Pol([X/\cG], n)$. We will now present a different way to compute $\Pol([X/\cG], 1)$.

\begin{lm}
Let $f\colon X\rightarrow Y$ be an epimorphism of derived stacks. Suppose $\cE_X$ and $\cE_Y$ are perfect complexes on $X$ and $Y$ respectively together with a map $f^*\cE_Y\rightarrow \cE_X$ such that the homotopy fiber of the dual map $\cE_X^*\rightarrow f^*\cE_Y^*$ is connective. Then the map
\[\Spec \Sym_{\cO_X}(\cE_X)\longrightarrow \Spec \Sym_{\cO_Y}(\cE_Y)\]
is an epimorphism.
\label{lm:vectorbundleepimorphism}
\end{lm}
\begin{proof}
The map $f\colon X\rightarrow Y$ being an epimorphism means that for any derived affine scheme $S$ and a map $g\colon S\rightarrow Y$ there is an \'{e}tale cover $p\colon \tilde{S}\rightarrow S$ and a commutative diagram
\[
\xymatrix{
\tilde{S} \ar^{p}[d] \ar^{\tilde{g}}[r] & X \ar^{f}[d] \\
S \ar^{g}[r] & Y
}
\]

We have to show that the same claim is true for
\[
E^*_X = \Spec\Sym_{\cO_X}(\cE_X)\longrightarrow E^*_Y=\Spec \Sym_{\cO_Y}(\cE_Y).
\]

Let $S$ be a derived affine scheme. A morphism $S\rightarrow E^*_Y$ is the same as a pair $(g, s)$ of a morphism $g\colon S\rightarrow Y$ and an element $s\in \Map_{\QCoh(S)}(g^*\cE_Y, \cO_S)$. Since $f\colon X\rightarrow Y$ is an epimorphism, we can lift $g$ to $\tilde{g}\colon \tilde{S}\rightarrow X$ as above. Let $K$ be the homotopy fiber of $\cE^*_X\rightarrow f^*\cE^*_Y$. Since it is connective, so is $\Gamma(\tilde{S}, \tilde{g}^* K)$. From the fiber sequence
\[\Gamma(\tilde{S}, \tilde{g}^* K)\longrightarrow \Gamma(\tilde{S}, \tilde{g}^* \cE^*_X)\longrightarrow \Gamma(\tilde{S}, \tilde{g}^*f^* \cE^*_Y)\]
we deduce that the projection
\[\Map_{\QCoh(\tilde{S})}(\tilde{g}^*\cE_X, \cO_{\tilde{S}})\longrightarrow \Map_{\QCoh(\tilde{S})}(\tilde{g}^* f^*\cE_Y, \cO_{\tilde{S}})\cong \Map_{\QCoh(\tilde{S})}(p^*g^*\cE_Y, \cO_{\tilde{S}})\]
is essentially surjective and hence $p^*s$ admits a lift.
\end{proof}

\begin{cor}
Let $n\geq 1$. Let $L^*[n-1] = \Spec \Sym_{\cO_X}(\cL[1-n])\rightarrow X$ be the total space of the perfect complex $\cL^*[n-1]$. The cotangent stack $\T^*[n]([X/\cG])$ is equivalent to the quotient of the smooth groupoid
\[\T^*[n-1](\cG)\rightrightarrows L^*[n-1].\]
\label{cor:cotangentgroupoid}
\end{cor}
\begin{proof}
Consider the map $X\rightarrow [X/\cG]$. Its $n$-shifted conormal bundle gives a morphism
\[L^*[n-1]\rightarrow \T^*[n]([X/\cG]).\]

Let $Y = [X/\cG]$, $\cE_X = \cL[n-1]$ and $\cE_Y = \bT_{[X/\cG]}[-n]$. The relative tangent complex of $X\rightarrow [X/\cG]$ fits into a fiber sequence
\[\cL\longrightarrow \bT_X\longrightarrow f^*\bT_{[X/\cG]}\]
which gives the required morphism $f^*\bT_{[X/\cG]}\rightarrow \cL[1]$. Taking the duals and rotating, we obtain a fiber sequence
\[\bL_X[n-1]\longrightarrow \cL^*[n-1]\longrightarrow f^*\bL_{[X/\cG]}[n],\]
so the fiber of $\cE^*_X\rightarrow f^*\cE^*_Y$ is equivalent to $\bL_X[n-1]$ which is connective since $n\geq 1$. Therefore, by \cref{lm:vectorbundleepimorphism} the morphism
\[L^*[n-1]\longrightarrow \T^*[n]([X/\cG])\]
is an epimorphism and hence is equivalent to the geometric realization of its nerve.

For any map $X_1\rightarrow X_2$ of derived stacks, the self-intersection of the conormal bundle $\N^*(X_1/X_2)\rightarrow \T^* X_2$ is equivalent to $\T^*[-1](X_1\times_{X_2} X_1)$. In our case we get
\[L^*[n-1]\times_{\T^*[n]([X/\cG])} L^*[n-1]\cong \T^*[n-1](\cG)\]
which gives the required groupoid.
\end{proof}

\begin{prop}
The space $\MC(\Pol^2([X/\cG], 1)[2])$ of Maurer--Cartan elements in the abelian dg Lie algebra $\Pol^2([X/\cG], 1)[2]$ is equivalent to the groupoid $\MultBivec(\cG)$ of multiplicative bivector fields on $\cG$.
\label{prop:XGbivectors}
\end{prop}
\begin{proof}
The complex $\Pol([X/\cG], 1)$ is given by the algebra of functions on $\T^*[2][X/\cG]$. By \cref{cor:cotangentgroupoid} we can compute it as a totalization of
\[\cosimp{\cO(L^*[1])}{\cO(\T^*[1]\cG)}\]

Thus, $\MC(\Pol^2([X/\cG], 1)[2])$ is the totalization of
\[\cosimp{\MC(\Gamma(X, \wedge^2\cL))}{\MC(\Gamma(\cG, \wedge^2 \T_\cG))}\]

Here the complexes in each term are concentrated in cohomological degree 0, so by \cref{prop:MCDeligne} we get that $\MC(\Gamma(X, \wedge^2\cL)) = \ast / \Gamma(X, \wedge^2\cL)$, i.e. the one-object groupoid with $\Gamma(X, \wedge^2\cL)$ as the automorphism group of the object, and similarly for $\MC(\Gamma(\cG, \wedge^2 \T_\cG))$.

By \cite[Proposition 2.7]{IPLGX} the multiplicativity condition for a bivector on $\cG$ is equivalent to the condition that the corresponding linear function on the groupoid $\T^*\cG\times_\cG \T^*\cG\rightrightarrows L^*\times_X L^*$ is a 1-cocycle. Therefore, using \cref{lm:groupoidTot} we conclude that the totalization has objects given by multiplicative bivectors on $\cG$ and morphisms by elements $\lambda\in\Gamma(X, \wedge^2\cL)$.
\end{proof}

Observe that \cref{thm:1shiftedBg,thm:1shiftedBG} imply that, if $G$ is connected, we have a Cartesian diagram of spaces
\[
\xymatrix{
\Pois(\B G, 1) \ar[r] \ar[d] & \Pois(\B \g, 1) \ar[d] \\
\MC(\C^\bullet(G, \wedge^2\g)) \ar[r] & \MC(\C^\bullet(\g, \wedge^2\g))
}
\]
where at the bottom we consider Maurer--Cartan spaces of abelian dg Lie algebras. In other words, a 1-shifted Poisson structure on $\B G$ is the same as a 1-shifted Poisson structure on $\B\g$ whose bivector is integrable. We will now show that the same statement is true in the groupoid setting. Let us first state some preliminary lemmas.

Given a graded dg Lie algebra $\g$ we denote by $\g(n)$ its weight $n$ component.

\begin{lm}
Suppose $\g_1\rightarrow \g_2$ is a morphism of graded dg Lie algebras such that for any $n\geq 3$ the fiber of $\g_1(n)\rightarrow \g_2(n)$ is concentrated in cohomological degrees $\geq 3$.

Then the diagram of spaces
\[
\xymatrix{
\MC(\g_1^{\geq 2}) \ar[r] \ar[d] & \MC(\g_2^{\geq 2}) \ar[d] \\
\MC(\g_1(2)) \ar[r] & \MC(\g_2(2)).
}
\]
is Cartesian.
\label{lm:MCCartesiandiagram}
\end{lm}
\begin{proof}
We will use the obstruction-theoretic argument given in \cite[Section 1.4]{Pri1}. By definition
\[\MC(\g_1^{\geq 2}) = \lim_n \MC(\g_1^{\geq 2} / \g_1^{\geq n})\]
and the result follows once we prove that
\[
\xymatrix{
\MC(\g_1^{\geq 2} / \g_1^{\geq n}) \ar[r] \ar[d] & \MC(\g_2^{\geq 2} / \g_2^{\geq n}) \ar[d] \\
\MC(\g_1(2)) \ar[r] & \MC(\g_2(2))
}
\]
is Cartesian for any $n\geq 3$. We prove the claim by induction. Indeed, the claim is tautologically true for $n=3$.

We prove the inductive step as follows. We have a sequence of simplicial sets
\[
\MC(\g_i^{\geq 2} / \g_i^{\geq (n+1)})\longrightarrow \MC(\g_i^{\geq 2} / \g_i^{\geq n})\longrightarrow \MC(\g_i(n)[-1]),
\]
where we treat $\g_i(n)[-1]$ as an abelian dg Lie algebra.

By \cite[Proposition 1.29]{Pri1} the homotopy fiber of $\MC(\g_i^{\geq 2} / \g_i^{\geq n})\rightarrow \MC(\g_i(n)[-1])$ at $0\in \MC(\g_i(n)[-1])$ is equivalent to $\MC(\g_i^{\geq 2} / \g_i^{\geq (n+1)})$. By assumptions the map
\[\MC(\g_1(n)[-1])\rightarrow \MC(\g_2(n)[-1])\]
is an equivalence of connected components of the trivial Maurer--Cartan element for $n\geq 3$. Therefore, the square
\[
\xymatrix{
\MC(\g_1^{\geq 2} / \g_1^{\geq (n+1)}) \ar[r] \ar[d] & \MC(\g_2^{\geq 2} / \g_2^{\geq (n+1)}) \ar[d] \\
\MC(\g_1^{\geq 2} / \g_1^{\geq n}) \ar[r] & \MC(\g_2^{\geq 2} / \g_2^{\geq n})
}
\]
is Cartesian which proves the inductive step.
\end{proof}

We are going to apply the previous statement to the map $\Pol([X/\cG], 1)\rightarrow \Pol([X/\cL], 1)$. First, recall the following basic fact, see e.g. \cite[Theorem 4]{Crainic}.

\begin{lm}
Let $\cG$ be a source-connected groupoid and $V$ a $\cG$-representation.
\begin{enumerate}
\item The restriction map
\[\H^0(\cG, V)\longrightarrow \H^0(\cL, V)\]
is an isomorphism.

\item The restriction map
\[\H^1(\cG, V)\longrightarrow \H^1(\cL, V)\]
is injective.
\end{enumerate}
\label{lm:groupoidvanest}
\end{lm}

The same statement is true for hypercohomology.

\begin{cor}
Let $\cG$ be a source-connected groupoid and $V^\bullet$ a complex of $\cG$-representations concentrated in non-negative degrees.

\begin{enumerate}
\item The restriction map
\[\H^0(\cG, V^\bullet)\longrightarrow \H^0(\cL, V^\bullet)\]
is an isomorphism.

\item The restriction map
\[\H^1(\cG, V^\bullet)\longrightarrow \H^1(\cL, V^\bullet)\]
is injective.
\end{enumerate}
\label{cor:fiberVanEst}
\end{cor}
\begin{proof}
In degree 0 we have
\[
\xymatrix{
\H^0(\cG, V^\bullet) \ar^{\sim}[r] \ar[d] & \H^0(\cG, \H^0(V^\bullet)) \ar^{\sim}[d] \\
\H^0(\cL, V^\bullet) \ar^{\sim}[r] & \H^0(\cL, \H^0(V^\bullet))
}
\]
so the map $\H^0(\cG, V^\bullet)\rightarrow \H^0(\cL, V^\bullet)$ is an isomorphism.

The hypercohomology spectral sequence gives long exact sequences
\[
\xymatrix{
0 \ar[r] & \H^1(\cG, \H^0(V^\bullet)) \ar@{^{(}->}[d] \ar[r] & \H^1(\cG, V^\bullet) \ar[d] \ar[r] & \H^0(\cG, \H^1(V^\bullet)) \ar^{\sim}[d] \ar[r] & \H^2(\cG, \H^0(V^\bullet)) \ar[d] \\
0 \ar[r] & \H^1(\cL, \H^0(V^\bullet)) \ar[r] & \H^1(\cL, V^\bullet) \ar[r] & \H^0(\cL, \H^1(V^\bullet)) \ar[r] & \H^2(\cL, \H^0(V^\bullet)).
}
\]
and the injectivity of $\H^1(\cG, V^\bullet)\rightarrow \H^1(\cL, V^\bullet)$ follows from the five lemma.
\end{proof}

\begin{lm}
The dg category $\Perf([X/\cG])$ is equivalent to full subcategory of the derived category $\cG$-representations consisting of perfect complexes.
\label{lm:PerfXG}
\end{lm}
\begin{proof}
The dg category $\QCoh([X/\cG])$ has a natural $t$-structure and let us denote by
\[\QCoh^+([X/\cG])\subset \QCoh([X/\cG])\]
the full subcategory of bounded below complexes. By \cite[Chapter I.3, Proposition 2.4.3]{GR1} we can identify $\QCoh^+([X/\cG])$ with the bounded below derived category of $\cG$-representations.  The claim then follows by passing to the subcategory of perfect complexes.
\end{proof}

\begin{lm}
Suppose $\cG$ is source-connected. Then the natural diagram of groupoids
\[
\xymatrix{
\QPoisGpd(\cG) \ar[r] \ar[d] & \QLieAlgd(\cL) \ar[d] \\
\MultBivec(\cG) \ar[r] & \MultBivec(\cL)
}
\]
is Cartesian.
\label{lm:connectedQPoisGpd}
\end{lm}
\begin{proof}
The map $\QLieAlgd(\cL)\rightarrow \MultBivec(\cL)$ is a fibration of groupoids (i.e. an isofibration), so it is enough to show that the map $\QPoisGpd(\cG)\rightarrow \QLieAlgd(\cL)\times_{\MultBivec(\cL)} \MultBivec(\cG)$ is an equivalence where we consider the strict fiber product. This map is clearly fully faithful.

Objects of both groupoids are pairs $(\Pi, \Omega)$ where $\Pi\in\Gamma(\cG, \wedge^2\T_\cG)$ and $\Omega\in\Gamma(X, \wedge^3\cL)$ satisfying certain equations. For $\QPoisGpd(\cG)$ we have the following three equations:
\begin{itemize}
\item The bivector $\Pi$ on $\cG$ is multiplicative.

\item We have an equality of multiplicative trivectors $\frac{1}{2}[\Pi, \Pi] = \Omega^R - \Omega^L$ on $\cG$.

\item We have an equality of fourvectors $[\Pi, \Omega^R] = 0$ on $\cG$.
\end{itemize}

For $\QLieAlgd(\cL)\times_{\MultBivec(\cL)} \MultBivec(\cG)$ we have the following three equations:
\begin{itemize}
\item The bivector $\Pi$ on $\cG$ is multiplicative.

\item Let $\delta$ be the 2-differential on $\cL$ corresponding to $\Pi$. Then we have an equality of 3-differentials $\frac{1}{2}[\delta, \delta] = \ad(\Omega)$.

\item We have $\delta(\Omega) = 0$.
\end{itemize}

There is a natural map from the graded Lie algebra of multiplicative multivectors on $\cG$ to the graded Lie algebra of differentials on $\cG$. Moreover, if $\cG$ is source-connected, it is shown in \cite[Proposition 2.35]{IPLGX} that it is injective. Thus, the equation $\frac{1}{2}[\Pi, \Pi] = \Omega^R - \Omega^L$ holds iff $\frac{1}{2}[\delta, \delta] = \ad(\Omega)$.

Moreover, \cite[Equation (18)]{IPLGX} implies that
\[[\Pi, \Omega^R] = (\delta \Omega)^R,\]
so $[\Pi, \Omega^R] = 0$ iff $\delta\Omega = 0$.

Thus, the map $\QPoisGpd(\cG)\rightarrow \QLieAlgd(\cL)\times_{\MultBivec(\cL)} \MultBivec(\cG)$ is an isomorphism on objects.
\end{proof}

\begin{thm}
Suppose $\cG\rightrightarrows X$ is a smooth algebraic groupoid which is source-connected. The space $\Pois([X/\cG], 1)$ of 1-shifted Poisson structures on $[X/\cG]$ is equivalent to the groupoid $\QPoisGpd(\cG)$ of quasi-Poisson structures on $\cG$.
\label{thm:quasipoissongroupoids}
\end{thm}
\begin{proof}
The map $f\colon [X/\cL]\rightarrow [X/\cG]$ is formally \'{e}tale. Therefore, $f^*\bT_{[X/\cG]}\cong \bT_{[X/\cL]}$. By \cref{lm:PerfXG} we may identify $\Perf([X/\cG])\cong \Rep \cG$ and let us denote by $\cE$ the complex of $\cG$-representations corresponding to $\bT_{[X/\cG]}$. Under the functor $\Rep \cG\rightarrow \Rep\cL$ it corresponds to a complex of $\cL$-representations which we also denote by $\cE$.

Let $p\colon X\rightarrow [X/\cG]$ be the projection. Then $p^*\bT_{[X/\cG]}$ fits into a fiber sequence
\[p^*\bT_{[X/\cG]}\longrightarrow \bT_{X/[X/\cG]}\cong \cL\longrightarrow \bT_X.\]
In particular, the $\cG$-representation $\cE$ is concentrated in degrees $\geq -1$.

We can identify
\[\Pol([X/\cG], 1)\cong \C^\bullet(\cG, \Sym(\cE[-2])),\qquad \Pol([X/\cL], 1)\cong \C^\bullet(\cL, \Sym(\cE[-2])).\]

Considering the weight $p$ part, we see that $\Sym(\cE[-2])$ is concentrated in degrees $\geq p$. Thus, by \cref{cor:fiberVanEst} the homotopy fiber of $\Pol^p([X/\cG], 1)\rightarrow \Pol^p([X/\cL], 1)$ is concentrated in degrees $\geq (p+2)$. Therefore, from \cref{lm:MCCartesiandiagram} we get that the diagram of spaces
\[
\xymatrix{
\MC(\Pol([X/\cG], 1)^{\geq 2}[2]) \ar[r] \ar[d] & \MC(\Pol([X/\cL], 1)^{\geq 2}[2]) \ar[d] \\
\MC(\Pol^2([X/\cG], 1)[2]) \ar[r] & \MC(\Pol^2([X/\cL], 1)[2])
}
\]
is Cartesian. Using \cref{thm:quasiliebialgebroids,cor:XLbivectors,prop:XGbivectors} we obtain a Cartesian diagram
\[
\xymatrix{
\Pois([X/\cG], 1) \ar[r] \ar[d] & \Pois([X/\cL], 1) \ar[d] \\
\MultBivec(\cG) \ar[r] & \MultBivec(\cL)
}
\]
Therefore, from \cref{lm:connectedQPoisGpd} we obtain an equivalence $\Pois([X/\cG], 1)\cong \QPoisGpd(\cG)$.
\end{proof}

\begin{cor}
With notations as before, the space $\Cois(X\rightarrow [X/\cG], 1)$ of 1-shifted coisotropic structures on the projection $X\rightarrow [X/\cG]$ is equivalent to the set of multiplicative Poisson bivectors $\Pi$ on $\cG$.
\label{cor:poissongroupoids}
\end{cor}
\begin{proof}
The proof is identical to the proof of \cref{cor:liebialgebroids} since $\Pol(X\rightarrow [X/\cG], 1)$ is the kernel of
\[\Pol([X/\cG], 1)\rightarrow \Pol(X, 1).\]
\end{proof}

\subsection{Quasi-symplectic groupoids}

One also has a symplectic analog of quasi-Poisson groupoids which are quasi-symplectic groupoids, i.e. groupoids $\cG\rightrightarrows X$ equipped with a multiplicative two-form $\omega$ on $\cG$ and a three-form $H$ on $X$ satisfying some closure and nondegeneracy equations, see \cite[Definition 2.5]{Xu}. The following statement is well-known.

\begin{prop}
Fix a groupoid $\cG\rightrightarrows X$. The space of 1-shifted symplectic structures on $[X/\cG]$ is equivalent to the following groupoid:
\begin{itemize}
\item Its objects are pairs $(\omega, H)$ endowing $\cG$ with a structure of a quasi-symplectic groupoid.

\item Morphisms $(\omega_1, H_1)\rightarrow (\omega_2, H_2)$ are given by two-forms $B\in\Omega^2(X)$ satisfying
\begin{align*}
\omega_2 - \omega_1 &= \d B \\
H_2 - H_1 &= \ddr B.
\end{align*}
\end{itemize}
\label{prop:quasisymplecticgroupoid}
\end{prop}
\begin{proof}
The space of 1-shifted presymplectic structures on $[X/\cG]$ is equivalent to the space $\cA^{2, cl}([X/\cG], 1)$ of closed two-forms on $[X/\cG]$ of degree $1$. The latter space is equivalent to the totalization of the cosimplicial space
\[
\cosimp{\cA^{2, cl}(X, 1)}{\cA^{2, cl}(\cG, 1)}
\]

If $Y$ is a smooth affine scheme, $\cA^{2, cl}(Y, 1)$ is equivalent to the following groupoid:
\begin{itemize}
\item Its objects are closed three-forms $H$ on $Y$,

\item Its morphisms from $H_1$ to $H_2$ are given by two-forms $B$ satisfying
\[\ddr B = H_2 - H_1.\]
\end{itemize}

Applying \cref{lm:groupoidTot} we deduce that the space of 1-shifted presymplectic structures on $[X/\cG]$ is equivalent to the following groupoid:
\begin{itemize}
\item Its objects are pairs $(\omega, H)$, where $\omega$ is a two-form on $\cG$ and $H$ is a closed three-form on $X$ satisfying
\begin{align*}
\d\omega &= 0 \\
\d H + \ddr \omega &= 0,
\end{align*}
where $\d$ is the \v{C}ech differential.

\item Its morphisms from $(\omega_1, H_1)$ to $(\omega_2, H_2)$ are given by two-forms $B$ on $X$ satisfying
\begin{align*}
\omega_2 - \omega_1 &= \d B \\
H_2 - H_1 &= \ddr B.
\end{align*}
\end{itemize}

The space of 1-shifted symplectic structures on $[X/\cG]$ is the subspace of the above space for which the two-form $\omega$ is non-degenerate. For $p\colon X\rightarrow [X/\cG]$ the natural projection, the pullback functor $p^*\colon \QCoh([X/\cG])\rightarrow \QCoh(X)$ is conservative, so we are looking for forms $\omega$ which induce a quasi-isomorphism
\[\omega^\sharp\colon p^*\bT_{[X/\cG]}\rightarrow p^*\bL_{[X/\cG]}[1]\]
i.e. a quasi-isomorphism
\[
\xymatrix{
(\cL \ar^{a}[r] \ar[d] & \T_X) \ar[d] \\
(\T^*_X \ar^{a^*}[r] & \cL^*)
}
\]
of complexes of quasi-coherent sheaves on $X$. By passing to the total complex, this condition is equivalent to the condition that
\[0\longrightarrow \cL\longrightarrow \T_X\oplus \T^*_X\longrightarrow \cL^*\longrightarrow 0\]
is an exact sequence, i.e. $\cL\rightarrow \T_X\oplus \T^*_X$ is a Lagrangian embedding. But the latter condition is known to be equivalent to the non-degeneracy of the quasi-symplectic groupoid, see \cite[Theorem 2.2]{BCWZ} where quasi-symplectic groupoids are called $H$-twisted presymplectic groupoids.
\end{proof}

We have a similar characterization of symplectic groupoids.

\begin{prop}
The space of pairs of a 1-shifted symplectic structure on $[X/\cG]$ and a Lagrangian structure on $X\rightarrow [X/\cG]$ is equivalent to the set of closed two-forms $\omega$ on $\cG$ endowing it with a structure of a symplectic groupoid.
\label{prop:symplecticgroupoid}
\end{prop}
\begin{proof}
Let us present $\Omega^\epsilon([X/\cG])$ as the totalization of the cosimplicial graded mixed cdga
\[\cosimp{\Omega^\epsilon(X)}{\Omega^\epsilon(\cG)}\]

The pullback $\Omega^\epsilon([X/\cG])\rightarrow \Omega^\epsilon(X)$ is surjective, so its homotopy fiber coincides with the strict fiber. Therefore, repeating the argument of \cref{prop:quasisymplecticgroupoid} we see that the space of pairs of a 1-shifted presymplectic structure on $[X/\cG]$ and an isotropic structure on $X\rightarrow [X/\cG]$ is equivalent to the set of closed multiplicative two-forms $\omega$ on $\cG$. Note that the two-form $\omega$ can be explicitly constructed by considering the isotropic self-intersection $X\times_{[X/\cG]} X\cong \cG$ which carries a presymplectic structure of shift $0$.

The nondegeneracy condition on $f\colon X\rightarrow [X/\cG]$ implies that the 1-shifted presymplectic structure on $[X/\cG]$ is 1-shifted symplectic since then the two vertical morphisms in
\[
\xymatrix{
\cL \ar[r] \ar[d] & \T_X \ar[d] \\
\T^*_X \ar[r] & \cL^*
}
\]
are isomorphisms which implies that the total morphism $f^*\bT_{[X/\cG]}\rightarrow f^*\bL_{[X/\cG]}[-1]$ is a quasi-isomorphism. Thus, it is enough to work out the nondegeneracy condition on the isotropic structure on $f$.

The nondegeneracy condition on the Lagrangian is equivalent to the fact that the natural morphism
\[\bT_{X/[X/\cG]}\longrightarrow \bL_X\]
is an equivalence. We can identify $\bT_{X/[X/\cG]}\cong \N_{X/\cG}\cong \cL$, where $\cL$ is the Lie algebroid of $\cG$ and $\N_{X/\cG}$ is the normal bundle to the unit section. The composite $\N_{X/\cG}\rightarrow \bL_X$ is identified with the natural morphism induced by the fact that the unit section $X\rightarrow \cG$ is isotropic. Thus, the isotropic structure on $X\rightarrow [X/\cG]$ is nondegenerate iff the unit section is in fact Lagrangian which is equivalent to the nondegeneracy condition of symplectic groupoids.
\end{proof}

\section{Symplectic realizations}

\subsection{Definitions}

\label{sect:nondegeneracy}

Let $X$ be a derived Artin stack. We denote by $\Symp(X, n)$ the space of $n$-shifted symplectic structures on $X$ as defined in \cite{PTVV}. Similarly, given a morphism $f\colon L\rightarrow X$ of such stacks we denote by $\Lagr(f, n)$ the space of $n$-shifted Lagrangian structures, i.e. pairs of an $n$-shifted symplectic structure on $X$ and a Lagrangian structure on $f$.

An $n$-shifted Poisson structure $\pi_X$ on $X$ induces a morphism $\pi_X^\sharp\colon\bL_X\rightarrow \bT_X[-n]$ given by contraction with the bivector. Similarly, an $n$-shifted coisotropic structure on $f\colon L\rightarrow X$ induces vertical morphisms of fiber sequences
\[
\xymatrix{
\bL_{L/X}[-1] \ar[r] \ar^{\pi_L^\sharp}[d] & f^*\bL_X \ar[r] \ar^{\pi_X^\sharp}[d] & \bL_L \ar[d] \\
\bT_L[-n] \ar[r] & f^*\bT_X[-n] \ar[r] & \bT_{L/X}[-n+1]
}
\]

\begin{defn} $ $
\begin{itemize}
\item An $n$-shifted Poisson structure on $X$ is \defterm{nondegenerate} if the induced morphism $\pi^\sharp_X\colon\bL_X\rightarrow \bT_X[-n]$ is a quasi-isomorphism. Denote by $\Pois^{nd}(X, n)\subset \Pois(X, n)$ the subspace of such nondegenerate structures.

\item An $n$-shifted coisotropic structure on $f\colon L\rightarrow X$ is \defterm{nondegenerate} if the $n$-shifted Poisson structure on $X$ is so and the induced morphism $\pi_L^\sharp\colon\bL_{L/X}\rightarrow \bT_L[1-n]$ is a quasi-isomorphism. Denote by $\Cois^{nd}(f, n)\subset \Cois(f, n)$ the space of subspace of such nondegenerate structures.
\end{itemize}
\end{defn}

The following statement was proved in \cite[Theorem 3.2.4]{CPTVV} and \cite[Theorem 3.33]{Pri1}.

\begin{thm}
One has an equivalence of spaces
\[\Pois^{nd}(X, n)\cong \Symp(X, n).\]

Under this equivalence the two-form $\omega^\sharp\colon\bT_X\xrightarrow{\sim} \bL_X[n]$ is inverse to the bivector $\pi^\sharp_X\colon\bL_X\xrightarrow{\sim} \bT_X[-n]$.
\label{thm:nondegeneratepoisson}
\end{thm}

An $n$-shifted Lagrangian structure on $f\colon L\rightarrow X$ induces vertical isomorphisms of fiber sequences
\[
\xymatrix{
\bT_L \ar[r] \ar^{\omega_L^\sharp}_{\sim}[d] & f^*\bT_X \ar[r] \ar^{\omega_X^\sharp}_{\sim}[d] & \bT_{L/X}[1] \ar_{\sim}[d] \\
\bL_{L/X}[n-1] \ar[r] & f^*\bL_X[n] \ar[r] & \bL_L[n]
}
\]

The following statement was proved in \cite{Pri2} for $n=0$ and \cite[Theorem 4.22]{MS2} for all $n$.

\begin{thm}
One has an equivalence of spaces
\[\Cois^{nd}(f, n)\cong \Lagr(f, n).\]

Under this equivalence the quasi-isomorphism $\omega^\sharp_L\colon\bT_L\xrightarrow{\sim} \bL_{L/X}[n-1]$ induced by the Lagrangian structure is inverse to the quasi-isomorphism $\pi^\sharp_L\colon\bL_{L/X}\xrightarrow{\sim} \bT_L[1-n]$ induced by the coisotropic structure.
\label{thm:nondegeneratecoisotropic}
\end{thm}

In particular, by \cref{thm:nondegeneratecoisotropic} we get a forgetful map
\[\Lagr(f, n)\cong \Cois^{nd}(f, n)\longrightarrow \Pois(L, n-1).\]

\begin{defn}
Suppose $L$ is a derived Artin stack equipped with an $(n-1)$-shifted Poisson structure. Its \defterm{symplectic realization} is the data of an $n$-shifted Lagrangian $L\rightarrow X$ whose induced $(n-1)$-shifted Poisson structure on $L$ coincides with the original one.
\end{defn}

Next, we are going to discuss symplectic realizations of coisotropic morphisms. Suppose $f_1\colon L_1\rightarrow X$ and $f_2\colon L_2\rightarrow X$ are two morphisms. In \cite[Theorem 2.9]{PTVV} the authors construct a map of spaces
\[\Lagr(f_1, n)\times_{\Symp(X, n)} \Lagr(f_2, n)\longrightarrow \Symp(L_1\times_X L_2, n-1),\]
i.e. an intersection of $n$-shifted Lagrangians carries an $(n-1)$-shifted symplectic structure.

Similarly, in \cite[Theorem 3.6]{MS2} the authors construct a map of spaces
\[\Cois(f_1, n)\times_{\Pois(X, n)} \Cois(f_2, n)\longrightarrow \Pois(L_1\times_X L_2, n-1),\]
i.e. an intersection of $n$-shifted coisotropics carries an $(n-1)$-shifted Poisson structure.

\begin{conjecture}
Coisotropic intersections have the following properties:
\begin{enumerate}
\item Suppose $L_1\rightarrow X$ and $L_2\rightarrow X$ carry a \emph{nondegenerate} $n$-shifted coisotropic structure. Then the $(n-1)$-shifted Poisson structure on $L_1\times_X L_2$ is nondegenerate.

\item The diagram of spaces
\[
\xymatrix{
\Cois^{nd}(f_1, n)\times_{\Pois^{nd}(X, n)} \Cois^{nd}(f_2, n) \ar[r] \ar^{\sim}[d] & \Pois^{nd}(L_1\times_X L_2, n-1) \ar^{\sim}[d] \\
\Lagr(f_1, n)\times_{\Symp(X, n)} \Lagr(f_2, n) \ar[r] & \Symp(L_1\times_X L_2, n-1)
}
\]
is commutative.
\end{enumerate}
\label{conj:intersections}
\end{conjecture}

\begin{defn}
A diagram of derived Artin stacks
\[
\xymatrix{
& C \ar[dl] \ar[dr] & \\
\tilde{X} \ar[dr] && X \ar[dl] \\
& Y &
}
\]
is an \defterm{$n$-shifted Lagrangian correspondence} provided we have the following data:
\begin{itemize}
\item An $n$-shifted symplectic structure on $Y$,

\item Lagrangian structures on $X\rightarrow Y$ and $\tilde{X}\rightarrow Y$,

\item A Lagrangian structure on $C\rightarrow \tilde{X}\times_Y X$, where $\tilde{X}\times_Y X$ is equipped with an $(n-1)$-shifted symplectic structure as a Lagrangian intersection.
\end{itemize}
\end{defn}

Suppose we have an $n$-shifted Lagrangian correspondence as above. In particular, $X$ carries an $(n-1)$-shifted Poisson structure and $X\rightarrow Y$ is its symplectic realization. Assuming \cref{conj:intersections}, we get an $(n-1)$-shifted Poisson structure on the projection $\tilde{X}\times_Y X\rightarrow X$. Since $C\rightarrow \tilde{X}\times_Y X$ carries an $(n-1)$-shifted coisotropic structure and $\tilde{X}\times_Y X\rightarrow X$ is an $(n-1)$-shifted Poisson morphism, the composite
\[C\longrightarrow \tilde{X}\times_Y X\longrightarrow X\]
also carries an $(n-1)$-shifted coisotropic structure.

\begin{defn}
Suppose $C\rightarrow X$ is a morphism of derived Artin stacks equipped with an $(n-1)$-shifted coisotropic structure. Its \defterm{symplectic realization} is the data of an $n$-shifted Lagrangian correspondence
\[
\xymatrix{
& C \ar[dl] \ar[dr] & \\
\tilde{X} \ar[dr] && X \ar[dl] \\
& Y &
}
\]
whose underlying $(n-1)$-shifted coisotropic structure on $C\rightarrow X$ coincides with the original one.
\end{defn}

\begin{remark}
Note that the picture of an $n$-shifted Lagrangian correspondence is completely symmetric and we also obtain an $(n-1)$-shifted coisotropic structure on $C\rightarrow \tilde{X}$. In a sense, the two $(n-1)$-shifted coisotropic structures $C\rightarrow X$ and $C\rightarrow \tilde{X}$ are dual. We refer to \cref{remark:PoissonLieDuality} for an example of this.
\end{remark}

Given a 2-shifted Lagrangian correspondence as above, using \cite[Corollary 2.15]{ABB} we obtain a 2-shifted Lagrangian correspondence
\[
\xymatrix{
& C\times_X C \ar[dl] \ar[dr] & \\
\tilde{X} \ar[dr] && \tilde{X} \ar[dl] \\
& Y &
}
\]

Therefore, the $(n-1)$-shifted Lagrangian $C\times_X C\rightarrow \tilde{X}\times_Y \tilde{X}$ gives a symplectic realization of the $(n-2)$-shifted Poisson structure on $C\times_X C$.

\subsection{Shift 0}

Our first goal is to relate our notion of symplectic realization to more classical notions in the case of ordinary (i.e. unshifted) Poisson structures.

Suppose $\cG\rightrightarrows X$ is a smooth symplectic groupoid. Since the unit section $X\rightarrow \cG$ is Lagrangian, we get an identification $\N_{X/\cG}\cong \T^*_X$. But the normal bundle to the unit section is canonically identified with the Lie algebroid $\cL$ of $\cG$. Thus, we obtain an isomorphism of vector bundles $\cL\cong \T^*_X$. By \cite[Proposition II.2.1]{CDW}, the induced Poisson structure has bivector given by the composite
\[\T^*_X\cong \cL\xrightarrow{a} \T_X.\]

In \cref{prop:symplecticgroupoid} we have shown that the data of a symplectic groupoid is equivalent to the data of a 1-shifted Lagrangian $X\rightarrow [X/\cG]$ which gives an a priori different $0$-shifted Poisson structure on $X$.

\begin{prop}
Suppose $\cG\rightrightarrows X$ is a smooth symplectic groupoid. The underlying Poisson structure on $X$ coincides with the 0-shifted Poisson structure underlying the 1-shifted Lagrangian $X\rightarrow [X/\cG]$.
\end{prop}
\begin{proof}
By \cref{thm:nondegeneratecoisotropic} the underlying nondegenerate 1-shifted coisotropic structure on $X\rightarrow [X/\cG]$ has the morphism $\pi^\sharp\colon \T^*_X\rightarrow \bT_{X/[X/\cG]}$ inverse to the quasi-isomorphism $\omega^\sharp\colon \bT_{X/[X/\cG]}\xrightarrow{\sim}\T^*_X$ induced by the Lagrangian structure. The latter morphism is given by the natural isomorphism $\cL\cong \T^*_X$ for the Lie algebroid $\cL$ underlying the symplectic groupoid $\cG$.

The 0-shifted Poisson structure on $X$ underlying the 1-shifted coisotropic structure on $X\rightarrow [X/\cG]$ has its bivector given by the composite
\[\T^*_X\rightarrow \bT_{X/[X/\cG]}\rightarrow \T_X.\]
But it is easy to see that $\cL\cong \bT_{X/[X/\cG]}\rightarrow \T_X$ coincides with the anchor map which gives the result.
\end{proof}

\begin{remark}
Note that the notion of a symplectic realization of a Poisson manifold introduced in \cite{We} is slightly more general. Namely, if $M$ is a Poisson manifold, its symplectic realization was defined to be a symplectic manifold $S$ with a submersive Poisson map $S\rightarrow M$. In particular, $S$ was not required to be a groupoid. However, symplectic groupoids integrating $M$ provide particularly nice examples of symplectic realizations.
\label{remark:symplecticrealization}
\end{remark}

\begin{example}
Let $G$ be a semisimple algebraic group, $P\subset G$ a parabolic subgroup and $M$ its Levi factor. Let $(-, -)$ be a nondegenerate symmetric bilinear $G$-invariant pairing on $\g=\Lie(G)$. By restriction it gives rise to a nondegenerate pairing on $\m=\Lie(M)$. Let us also denote $\p=\Lie(P)$. These pairings give rise to 2-shifted symplectic structures on $\B G$ and $\B M$ and it is shown in \cite{Sa2} that the correspondence
\[
\xymatrix{
& \B P \ar[dl] \ar[dr] & \\
\B M && \B G
}
\]
has a 2-shifted Lagrangian structure coming from the exact sequence
\begin{equation}
0\rightarrow \p\rightarrow \g\oplus \m\rightarrow \p^*\rightarrow 0
\label{eq:pexactsequence}
\end{equation}
of $P$-representations.

Now suppose $E$ is an elliptic curve. Denote by $\Bun_{-}(E) = \Map(E, \B(-))$ the moduli space of principal bundles on $E$. Then applying $\Map(E, -)$ to the above 2-shifted Lagrangian correspondence we obtain a 1-shifted Lagrangian correspondence
\[
\xymatrix{
& \Bun_P(E) \ar[dl] \ar[dr] & \\
\Bun_M(E) && \Bun_G(E)
}
\]
using the AKSZ theorem of \cite{PTVV}. In this way we obtain a 0-shifted Poisson structure on $\Bun_P(E)$. Let us work it out explicitly. Denote $L=\Bun_P(E)$ and $X=\Bun_M(E)\times \Bun_G(E)$ for brevity.

Fix a principal $P$-bundle $F\rightarrow E$. We can identify the tangent complex to $L$ at $F$ with
\[\bT_{L, F}\cong \Gamma(E, \ad F)[1].\]

Similarly, we can identify
\[\bT_{L/X, F}\cong \fib\left(\Gamma(E, F\times_P \p)\rightarrow \Gamma(E, F\times_P(\m\oplus\g))\right)[1].\]

Using the exact sequence \eqref{eq:pexactsequence} we can simplify it to
\[\bT_{L/X, F}\cong \Gamma(E, F\times_P \p^*).\]

Moreover, $\bL_{L, F}\cong \Gamma(E, F\times_P \p^*)$ and the equivalence $\bT_{L/X}\cong \bL_L$ induced by the 1-shifted Lagrangian structure on $L\rightarrow X$ is the identity.

The bivector underlying the 0-shifted Poisson structure on $L$ is given by the composite
\[\bL_L\xrightarrow{\sim} \bT_{L/X}\rightarrow \bT_L\]
which is easily seen to be given by the morphism
\[\Gamma(E, F\times_P \p^*)\rightarrow \Gamma(E, F\times_P\p)[1]\]
induced by the morphism $\p^*\rightarrow \p[1]$ in the derived category of $P$-representations obtained as the connecting homomorphism in the sequence \eqref{eq:pexactsequence}. In this way we recover the Feigin--Odesskii Poisson structure on $\Bun_P(E)$ constructed in \cite{FO}.

Note that in this way we also obtain a symplectic groupoid
\[\Bun_P(E)\times_{\Bun_M(E)\times \Bun_G(E)} \Bun_P(E)\rightrightarrows \Bun_P(E)\]
integrating the Feigin--Odesskii Poisson structure.
\end{example}

\subsection{Shift 1}

In this section we discuss symplectic realizations of some 1-shifted Poisson structures.

\begin{defn}
A \defterm{Manin pair} is a pair of finite-dimensional Lie algebras $\fd$ and $\g$, where $\fd$ is equipped with a nondegenerate invariant pairing and $\g\subset \fd$ is a Lagrangian subalgebra.
\end{defn}

Given a Manin pair $(\fd, \g)$ choose a Lagrangian splitting $\fd\cong \g\oplus \fd/\g$ where we do not assume that $\fd/\g\subset \fd$ is closed under the bracket. Formulas \eqref{eq:forgetqlie1} and \eqref{eq:forgetqlie2} then define a quasi-Lie bialgebra structure on $\g$.

\begin{remark}
Changing the Lagrangian complement to $\g\subset \fd$ corresponds to twists of the quasi-Lie bialgebra structure on $\g$.
\end{remark}

\begin{defn}
A \defterm{Manin triple} is a triple $(\fd, \g, \g^*)$ of finite-dimensional Lie algebras, where $\fd$ is equipped with a nondegenerate invariant pairing and $\g\subset \fd$ and $\g^*\subset \fd$ are complementary Lagrangian subalgebras.
\end{defn}

\begin{remark}
If $\g\subset \fd$ is a Manin pair, we obtain an exact sequence of $\g$-representations
\[0\longrightarrow \g\longrightarrow \fd\longrightarrow \g^*\longrightarrow 0\]
and the extra data in a Manin triple is given by a splitting of this exact sequence compatibly with the Lie brackets and the pairings. In particular, in a Manin triple $(\fd, \g, \g^*)$ $\g^*$ is indeed a linear dual of $\g$.
\end{remark}

Given a Manin triple $(\fd, \g, \g^*)$, the Lie bracket on $\g^*$ gives rise to a Lie cobracket on $\g$ which is easily seen to endow $\g$ with a structure of a Lie bialgebra.

Now suppose $D$ is an algebraic group whose Lie algebra $\fd$ is equipped with a nondegenerate $D$-invariant pairing. Also suppose $G\subset D$ is a closed subgroup whose Lie algebra $\g\subset \fd$ is Lagrangian. Such pairs are called \defterm{group pairs} in \cite{AKS}. By the results of \cite[Section 3.3]{AKS} one obtains a quasi-Poisson structure on $G$ such that the induced quasi-Lie bialgebra structure on $\g$ coincides with the one above. In particular, we get a 1-shifted Poisson structure on $\B G$ by \cref{thm:1shiftedBG}.

\begin{prop}
Suppose $(D, G)$ is a group pair as above. Then the 2-shifted Lagrangian
\[\B G\rightarrow \B D\]
is a symplectic realization of the 1-shifted Poisson structure on $\B G$ corresponding to the quasi-Poisson structure on $G$.
\label{prop:quasipoissonrealization}
\end{prop}
\begin{proof}
We have computed the space $\Cois(\B G\rightarrow \B D, 2)$ in \cref{prop:coisotropicsubgroup}. In particular, nondegenerate coisotropic structures are given by nondegenerate pairings on $D$ for which $\g$ is Lagrangian which provides the identification $\Cois^{nd}(\B G\rightarrow \B D, 2)\cong \Lagr(\B G\rightarrow \B D, 2)$.

The forgetful map $\Cois^{nd}(\B G\rightarrow \B D, 2)\rightarrow \Pois(\B G, 1)$ was computed in \cref{prop:coisotropicqlie} where it was shown that the underlying quasi-Lie bialgebra structure on $\g$ coincides with the one given by formulas \eqref{eq:forgetqlie1} and \eqref{eq:forgetqlie2}.
\end{proof}

Similarly, we may consider \defterm{group triples} $(D, G, G^*)$ which are triples of algebraic groups whose Lie algebras form a Manin triple. From such a data we get a Poisson-Lie structure on $G$ and hence a $1$-shifted coisotropic structure on $\pt\rightarrow \B G$ by \cref{cor:poissongroups}. The two morphisms $\B G\rightarrow \B D$ and $\B G^*\rightarrow \B D$ carry an obvious $2$-shifted Lagrangian structure. Moreover, their intersection $\B G^*\times_{\B D} \B G\cong \dquot{G^*}{D}{G}$ has the zero tangent complex at the unit of $D$ since $\g^*,\g\subset \fd$ are transverse. Therefore, the inclusion of the unit $\pt\rightarrow \dquot{G^*}{D}{G}$ has a unique $1$-shifted Lagrangian structure. In other words, we obtain a 2-shifted Lagrangian correspondence
\begin{equation}
\xymatrix{
& \pt \ar[dl] \ar[dr] & \\
\B G^* \ar[dr] && \B G \ar[dl] \\
& \B D &
}
\label{eq:PoissonLieDuality}
\end{equation}

\begin{prop}
Suppose $(D, G, G^*)$ is a group triple as above. Then the $2$-shifted Lagrangian correspondence \eqref{eq:PoissonLieDuality} is a symplectic realization of the $1$-shifted coisotropic structure on $\pt\rightarrow \B G$ corresponding to the Poisson-Lie structure on $G$.
\label{prop:poissonlierealziation}
\end{prop}
\begin{proof}
To prove the claim, we have to compare the Poisson-Lie structure on $G$ obtained from the Manin triple with the one obtained from the 2-shifted Lagrangian correspondence. For this it is enough to compare the induced Lie cobrackets on $\g$.

As observed in \cref{sect:nondegeneracy}, the symplectic realization of the Poisson structure on $G$ is given by the $1$-shifted Lagrangian $f\colon G\rightarrow \B G^*\times_{\B D} \B G^*\cong [\dquot{G^*}{D}{G^*}]$. To compute the induced Poisson structure on $G$, let us unpack the underlying two-forms on $G$ and $[\dquot{G^*}{D}{G^*}]$.

Trivialize the tangent bundle to $D$ and $G$ using left translations. Given an element $d\in D$, the tangent complex to $[\dquot{G^*}{D}{G^*}]$ at $d$ is
\[\bT_{[\dquot{G^*}{D}{G^*}], d}\cong (\g^*\oplus \g^*\rightarrow \fd)\]
in degrees $-1$ and $0$, where one of the maps is the identity and the other map is given by the composite
\[\g^*\longrightarrow \fd\xrightarrow{\Ad_d}\fd.\]

Using the exact sequence $0\rightarrow \g^*\rightarrow \fd\rightarrow \g\rightarrow 0$ we get a quasi-isomorphic model for the tangent complex as
\[\bT_{[\dquot{G^*}{D}{G^*}], d}\cong (\g^*\xrightarrow{\Ad_d} \g).\]
The $1$-shifted symplectic structure on $[\dquot{G^*}{D}{G^*}]$ induces an obvious isomorphism
\[\bT_{[\dquot{G^*}{D}{G^*}], d}\cong \bL_{[\dquot{G^*}{D}{G^*}], d}[1].\]

The Lagrangian structure on $G\rightarrow [\dquot{G^*}{D}{G^*}]$ gives rise to a fiber sequence of complexes
\[\bT_{G, g}\longrightarrow \bT_{[\dquot{G^*}{D}{G^*}], g}\longrightarrow \bL_{G, g}[1]\]
for $g\in G$. Explicitly, it is given by the vertical fiber sequence of complexes
\[
\xymatrix{
(0 \ar[d] \ar[r] & \g) \ar@{=}[d] \\
(\g^* \ar@{=}[d] \ar^{\Ad_g}[r] & \g) \ar[d] \\
(\g^* \ar[r] & 0)
}
\]

Therefore, the connecting homomorphism $\bL_{G, g}\rightarrow \bT_{G, g}$ is given by the morphism $\Ad_g\colon \g^*\rightarrow \g$ which is the underlying bivector on $G$. The Lie cobracket on $\g$ is given by linearizing this bivector at $g=e$, so we see that the cobracket is given by $[x, -]\colon \g^*\rightarrow \g$ for $x\in\g$, i.e. the map $\g\otimes \g^*\rightarrow \g$ given by the coadjoint action of $\g^*$ which recovers the Lie cobracket on $\g$ coming from the Manin triple.
\end{proof}

\begin{remark}
Note that in the course of the proof we have constructed a symplectic realization of the underlying Poisson-Lie structure on $G$ as $G\rightarrow [\dquot{G^*}{D}{G^*}]$. In particular, we obtain the symplectic groupoid
\[G\times_{[\dquot{G^*}{D}{G^*}]} G\cong (G\times G^*)\times_D (G\times G^*)\rightrightarrows G\]
known as the Lu--Weinstein groupoid \cite[Section 4.2]{Lu2}. A similar interpretation of the Lu--Weinstein groupoid was previously given in \cite{Sev}.
\end{remark}

\begin{remark}
The symmetry of the Lagrangian correspondence \eqref{eq:PoissonLieDuality} reflects Poisson-Lie duality.
\label{remark:PoissonLieDuality}
\end{remark}

\section{Classical \texorpdfstring{$r$}{r}-matrices}

In this section we relate shifted Poisson structures to the notion of a classical $r$-matrix.

\subsection{Classical notions}

Let $r\in\g\otimes \g$. The classical Yang--Baxter equation is an equation in $\g\otimes \g\otimes \g$ given by
\begin{equation}
\CYBE(r) := [r_{12}, r_{13}] + [r_{12}, r_{23}] + [r_{13}, r_{23}] = 0,
\label{eq:CYBE}
\end{equation}
where, for instance, $r_{12} = r\otimes 1\in (\U\g)^{\otimes 3}$.

Now consider an element $c\in\Sym^2(\g)^G$ and the induced trivector
\[\phi = -\frac{1}{6}[c_{12}, c_{23}].\]
The data of twist from the quasi-Poisson group $(G, \pi=0, \phi)$ to $(G, \pi, 0)$ is given by $\lambda\in\wedge^2(\g)$ satisfying
\[\frac{1}{2}\llbracket\lambda, \lambda\rrbracket = \phi,\qquad \pi = \lambda^L - \lambda^R.\]
In turn, the first equation is equivalent to the classical Yang--Baxter equation \eqref{eq:CYBE} for $r=2\lambda + c$.

\begin{defn}
A \defterm{quasi-triangular classical $r$-matrix} is the data of $r\in\g\otimes \g$ such that
\begin{enumerate}
\item $\CYBE(r) = 0$,

\item The symmetric part of $r$ is $G$-invariant.
\end{enumerate}
\end{defn}

One also has the following generalization. Let $H\subset G$ be a subgroup and $U\subset \h^*$ an open dense subset. Consider a function $r\colon U\rightarrow \g\otimes \g$. Its differential is $\ddr r\colon U\rightarrow \h\otimes \g\otimes \g$. We denote by $\Alt(\ddr r)$ the function $U\rightarrow \wedge^3(\g)$ obtained as its antisymmetrization.

\begin{defn}
A \defterm{quasi-triangular classical dynamical $r$-matrix} is the data of a function $r\colon U\rightarrow \g\otimes \g$ such that
\begin{enumerate}
\item $r$ is an $\h$-equivariant function,

\item $\CYBE(r) = 0$,

\item The symmetric part of $r$ is a constant element of $\Sym^2(\g)^G$.
\end{enumerate}
\end{defn}

Note that the data of a quasi-triangular classical dynamical $r$-matrix is equivalent to the data of $c\in\Sym^2(\g)^G$ and an $\h$-equivariant function $\lambda\in U\rightarrow \wedge^2(\g)$ such that
\[\frac{1}{2}\llbracket\lambda, \lambda\rrbracket + \Alt(\ddr\lambda) = \phi.\]

\subsection{Non-dynamical case}

\begin{prop}
Let $G$ be an algebraic group. The space parametrizing the pairs of
\begin{itemize}
\item A 2-shifted Poisson structure $\pi$ on $\B G$,

\item A 1-shifted Poisson morphism $\pt\rightarrow \B G$ with the 1-shifted Poisson structure on $\B G$ obtained from $\pi$
\end{itemize}
is equivalent to the set of quasi-triangular Poisson structures on $G$, i.e. the set of quasi-triangular classical $r$-matrices.

The quasi-triangular Poisson structure is factorizable iff the underlying 2-shifted Poisson structure on $\B G$ is 2-shifted symplectic.
\label{prop:constantrmatrix}
\end{prop}
\begin{proof}
Indeed, the space of 2-shifted Poisson structures on $\B G$ is equivalent to the set of Casimir elements $c\in\Sym^2(\g)^G$ by \cref{prop:2shiftedBG}. The induced 1-shifted Poisson structure on $\B G$ corresponds to the quasi-Poisson structure $(\pi=0, \phi)$, where
\[\phi = -\frac{1}{6}[c_{12}, c_{23}].\]

A 1-shifted Poisson morphism $\pt\rightarrow \B G$ is equivalent to the data of a Poisson structure $\pi$ on $G$ by \cref{cor:poissongroups}. The compatibility between the two is given by an element $\lambda\in\wedge^2(\g)$ twisting $(0, \phi)$ into $(\pi, 0)$, i.e. we get an equation
\[\frac{1}{2}\llbracket\lambda, \lambda\rrbracket = -\frac{1}{6}[c_{12}, c_{23}]\]
and hence $r = 2\lambda + c$ is a quasi-triangular classical $r$-matrix.
\end{proof}

\subsection{Dynamical case}

Before we give a description of dynamical $r$-matrices similar to \cref{prop:constantrmatrix} we need to construct a canonical 1-shifted Poisson structure on $[\g^*/G]$. We can identify $[\g^*/G]\cong \T^*[1](\B G)$, so by \cite{Cal2} it carries a 1-shifted symplectic structure. Replacing $G$ by the formal completion, it has the following description. We can identify
\[\Omega^\epsilon([\g^*/\hG])\cong \C^\bullet(\g, \Omega^\epsilon(\g^*)\otimes \Sym(\g^*[-2]))\cong \C^\bullet(\g, \Sym(\g\oplus \g[-1]\oplus \g^*[-2]),\]
where $\g^*[-2]$ and $\g[-1]$ are concentrated in weight $1$. The internal differential on $\Omega^\epsilon([\g^*/\hG])$ has the following two components:
\begin{itemize}
\item The Chevalley--Eilenberg differential,

\item The Cartan differential $\alpha\in\Omega^\epsilon(\g^*)\mapsto \iota_{a(-)}\alpha\in \Omega^\epsilon(\g^*)\otimes \g^*$, where $a\colon \g\rightarrow \T_{\g^*}$ is the action map.
\end{itemize}

We have a canonical Maurer--Cartan $\g^*$-valued one-form $\theta\in\Omega^1(\g^*; \g^*)$ on $\g^*$ which gives rise to a two-form on $[\g^*/\hG]$ of degree $1$. It is closed under the internal differential due to $\g$-equivariance. It is also closed under the de Rham differential since $\theta$ is constant on $\g^*$.

Similarly, we can identify
\[\Pol([\g^*/\hG], 1)\cong \C^\bullet(\g, \Sym(\g\oplus \g[-1]\oplus \g^*[-2]),\]
where $\g[-1]$ and $\g^*[-2]$ are again concentrated in weight $1$. The same element $\theta$ gives a degree 3 element $\pi\in\Pol([\g^*/\hG], 1)$. The Schouten bracket pairs $\g^*[-1]$ with $\g[-1]$ and $\g$ with $\g^*[-2]$, so $[\pi, \pi] = 0$. Therefore, $\pi$ defines a $1$-shifted Poisson structure on $[\g^*/\hG]$. The following statement is clear from the explicit formula for this 1-shifted Poisson structure.

\begin{lm}
The projection $[\g^*/G]\rightarrow \B G$ is a 1-shifted Poisson morphism where $\B G$ has the zero $1$-shifted Poisson structure.
\label{lm:coadjointprojection}
\end{lm}

Both $\omega$ and $\pi$ are clearly $G_{\dR}$-invariant, so they descend to a $1$-shifted symplectic and $1$-shifted Poisson structure on $[\g^*/G]$.

\begin{lm}
The $1$-shifted Poisson structure $\pi$ on $[\g^*/G]$ is nondegenerate and is compatible with the $1$-shifted symplectic structure $\omega$.
\end{lm}
\begin{proof}
$\pi$ induces a morphism $\bL_{[\g^*/G]}\rightarrow \bT_{[\g^*/G]}[-1]$ of complexes of quasi-coherent sheaves on $[\g^*/G]$ and to check that it is an equivalence it is enough to check it after pulling back to $\g^*$. But then $\pi$ induces a morphism
\[
\xymatrix{
(\g\otimes \cO_{\g^*} \ar^{\id}[d] \ar[r] & \g^*\otimes \cO_{\g^*}) \ar^{\id}[d] \\
(\g\otimes \cO_{\g^*} \ar[r] & \g^*\otimes \cO_{\g^*})
}
\]
of complexes of sheaves on $\g^*$ which is indeed an equivalence.

Being a $1$-shifted Poisson structure, $\pi$ endows $\Pol([\g^*/\hG], 1)$ with a graded mixed structure and by the universal property of the de Rham algebra, we obtain a morphism
\[\Omega^\epsilon([\g^*/\hG], 1)\longrightarrow \Pol([\g^*/\hG], 1)\]
which after unpacking corresponds to the identity
\[\C^\bullet(\g, \Sym(\g\oplus \g[-1]\oplus \g^*[-2]))\longrightarrow \C^\bullet(\g, \Sym(\g\oplus \g[-1]\oplus \g^*[-2])).\]
To show that $\pi$ is inverse to $\omega$, we have to produce a compatibility data between $\omega$ and $\pi$ in the sense of \cite[Definition 1.4.20]{CPTVV}. But since both elements are concentrated purely in weight $2$, this reduces to showing that $\omega^\sharp\colon \bT_{[\g^*/\hG]}\rightarrow \bL_{[\g^*/\hG]}[1]$ is inverse to $\pi^\sharp\colon \bL_{[\g^*/\hG]}\rightarrow \bT_{[\g^*/\hG]}[-1]$ (see e.g. \cite[Example 1.21]{Pri1}) which is obvious.
\end{proof}

From now on we will simply refer to $\pi$ as the canonical 1-shifted Poisson structure on $[\g^*/G]$.

\begin{prop}
Let $G$ be an algebraic group, $H\subset G$ a subgroup and $U\subset \h^*$ an open $H$-invariant subscheme. The space parametrizing the pairs of
\begin{itemize}
\item A 2-shifted Poisson structure $\pi$ on $\B G$,

\item A 1-shifted Poisson morphism $f\colon [U/H]\rightarrow \B G$ compatible with the canonical 1-shifted Poisson structure on $[U/H]\subset [\h^*/H]\cong \T^*[1](\B H)$ and the 1-shifted Poisson structure on $\B G$ induced from $\pi$
\end{itemize}
is equivalent to the set of quasi-triangular classical dynamical $r$-matrices with base $U$.
\label{prop:dynamicalrmatrix}
\end{prop}
\begin{proof}
It is enough to prove the corresponding statement where we replace groups by their formal completion. Let $g\colon [U/\h]\rightarrow [U/\h]\times \B\g$ be the graph. By \cref{thm:graphcoisotropic} we have to compute the pullback
\[
\xymatrix{
\Pois([U/\h]\rightarrow \B\g, 1) \ar[r] \ar[d] & \Pois(\B\g, 2) \ar[d] \\
\Cois(g, 1) \ar[r] & \Pois([U/\h]\times \B\g, 1).
}
\]

Here the map $\Pois(\B\g, 2)\rightarrow \Pois([U/\h]\times \B\g, 1)$ is given by sending a 2-shifted Poisson structure $c$ on $\B\g$ to the sum of the canonical 1-shifted Poisson structure on $[U/\h]\subset \T^*[1](\B \h)$ and the 1-shifted Poisson structure on $\B\g$ given by \cref{prop:forget21BG}.

We can identify
\[\Pol([U/\h]\times \B\g, 1)\cong \C^\bullet(\g, \Sym(\g[-1]))\otimes \C^\bullet(\h, \Sym(\h[-1]\oplus \h^*[-2])\otimes \cO(U))\]
and
\[\Pol([U/\h] / ([U/\h]\times \B\g), 0)\cong \C^\bullet(\h, \Sym(\g[-1])\otimes \cO(U)).\]

Moreover,
\[\Pol(g, 1)\cong \ker\left(\Pol([U/\h]\times \B\g, 1)\longrightarrow \Pol([U/\h] / ([U/\h]\times \B\g), 0)\right).\]

Let us begin by computing the space $\Pois([U/\h]\times \B\g, 1)$. Since $\Pol[U/\h]\times \B\g, 1)^{\geq 2}[2]$ is concentrated in non-negative degrees, by \cref{prop:MCDeligne} the space $\Pois([U/\h]\times \B\g, 1)$ is equivalent to the corresponding Deligne groupoid that we now proceed to compute.

Elements of $\Pol([U/\h]\times \B\g, 1)$ of degree 2 have the following type:
\begin{itemize}
\item $\lambda^1\in\wedge^2(\g)\otimes \cO(U)$,
\item $\lambda^2\in\g\otimes \h\otimes \cO(U)$,
\item $\lambda^3\in\wedge^2(\h)\otimes \cO(U)$.
\end{itemize}

Elements of $\Pol([U/\h]\times \B\g, 1)$ of degree 3 have the following type:
\begin{itemize}
\item $\alpha^1\in\wedge^2(\g)\otimes \g^*\otimes \cO(U)$,
\item $\alpha^2\in\wedge^2(\g)\otimes \h^*\otimes \cO(U)$,
\item $\alpha^3\in\g\otimes \h\otimes \g^*\otimes \cO(U)$,
\item $\alpha^4\in\g\otimes \h\otimes \h^*\otimes \cO(U)$,
\item $\alpha^5\in\wedge^2(\h)\otimes \g^*\otimes \cO(U)$,
\item $\alpha^6\in\wedge^2(\h)\otimes \h^*\otimes \cO(U)$,
\item $\alpha^7\in\g\otimes \h^*\otimes\cO(U)$,
\item $\alpha^8\in\h\otimes \h^*\otimes \cO(U)$,
\item $\beta^1\in\wedge^3(\g)\otimes \cO(U)$,
\item $\beta^2\in\wedge^2(\g)\otimes \h\otimes \cO(U)$,
\item $\beta^3\in\g\otimes \wedge^2(\h)\otimes \cO(U)$,
\item $\beta^4\in\wedge^3(\h)\otimes \cO(U)$.
\end{itemize}

Let $\rho\colon \h\rightarrow \h^*\otimes \cO(U)$ be the coadjoint action of $\h$ on $U$. Then we get the differential equations
\begin{align}
0 &= \frac{d\alpha^1}{dt} + \dCE^\g \lambda^1, &0 = \frac{d\alpha^2}{dt} + \dCE^\h \lambda^1 \nonumber \\
0 &= \frac{d\alpha^3}{dt} + \dCE^\g \lambda^2, &0 = \frac{d\alpha^4}{dt} + \dCE^\h \lambda^2 \nonumber \\
0 &= \frac{d\alpha^5}{dt}, &0 = \frac{d\alpha^6}{dt} + \dCE^\h \lambda^3 \nonumber \\
0 &= \frac{d\beta^1}{dt} + \rho(\lambda^2), &0 = \frac{d\beta^2}{dt} + \rho(\lambda^3)
\label{eq:dynamicalMCequations1}
\end{align}
and
\begin{align}
0 &= \frac{d\gamma^1}{dt} + [\alpha^1, \lambda^1] + [\beta^1, \lambda^1] + [\alpha^2, \lambda^2] \nonumber \\
0 &= \frac{d\gamma^2}{dt} + [\alpha^3, \lambda^1] + [\beta^2, \lambda^1] + [\alpha^1, \lambda^2] + [\alpha^4, \lambda^2] + [\beta^1, \lambda^2] + [\alpha^2, \lambda^3] \nonumber \\
0 &= \frac{d\gamma^3}{dt} + [\alpha^5, \lambda^1] + [\alpha^3, \lambda^2] + [\alpha^6, \lambda^2] + [\beta^2, \lambda^2] + [\alpha^4, \lambda^3] + [\beta^1, \lambda^3] \nonumber \\
0 &= \frac{d\gamma^4}{dt} + [\alpha^5, \lambda^2] + [\alpha^6, \lambda^3] + [\beta^2, \lambda^3].
\label{eq:dynamicalMCequations2}
\end{align}

Objects of $\Pois([U/\h]\times \B\g, 1)$ are given by Maurer--Cartan elements $\alpha^1,...,\alpha^6,\beta^1,\beta^2,\gamma^1,...,\gamma^4$. Morphisms are given by elements $\lambda^1,...,\lambda^3$ and solutions of the differential equations \eqref{eq:dynamicalMCequations1} and \eqref{eq:dynamicalMCequations2}.

The groupoid $\Cois(g, 1)$ is the subgroupoid of $\Pois([U/\h]\times \B\g, 1)$ whose objects are Maurer--Cartan elements as before satisfying
\[\sum_{i=1}^6 \alpha^i = 0,\qquad \sum_{i=1}^4 \gamma^i = 0\]
and whose morphisms are elements $\lambda_1,...,\lambda^3$ satisfying
\[\sum_{i=1}^3 \lambda_i = 0.\]

The space $\Pois([U/\h]\rightarrow \B\g, 1)$ is given by the homotopy fiber product of groupoids
\[\Pois([U/\h]\rightarrow \B\g, 1) = \Cois(g, 1) \times_{\Pois([U/\h]\times \B\g, 1)} \Pois(\B \g, 2)\]
for which we will use the standard model, see e.g. \cite[Lemma 2.2]{Ho2}. Let $\cG\subset \Pois([U/\h]\rightarrow \B\g, 1)$ be the full subgroupoid where we assume that $\lambda^2=0$ and $\lambda^3 = 0$ on objects. Then it's clear that objects in $\cG$ have no automorphisms and it is merely a set. Moreover, the inclusion is essentially surjective, hence is an equivalence. $\cG$ is therefore the set parametrizing elements $c\in\Sym^2(\g)^G$ and $\lambda^1\in\wedge^2(\g)\otimes \cO(U)$ satisfying some equations that we will now describe.

Let $\phi = -\frac{1}{6}[c_{12}, c_{23}]$, $\lambda^2 = 0$, $\lambda^3=0$. We consider solutions of \eqref{eq:dynamicalMCequations1} and \eqref{eq:dynamicalMCequations2} with the initial conditions $\beta^2(0)=\id_h\in\h^*\otimes \h$, $\gamma^1(0)=\phi$ and the rest of $\alpha^i(0),\beta^i(0),\gamma^i(0)$ zero. The nonzero equations are
\begin{align*}
\alpha^1(t) &= -t\dCE^\g\lambda^1 \\
\alpha^2(t) &= -t\dCE^\h\lambda^1 \\
\beta^2(t) &= \id_\h \\
\gamma^1(t) &= \phi + \frac{t^2}{2} [\dCE^\g\lambda^1, \lambda^1] \\
\gamma^2(t) &= -t[\id_\h, \lambda^1].
\end{align*}

Note that $[\id_\h, \lambda^1] = \ddr \lambda^1\in\h\otimes \wedge^2(\g)\otimes \cO(U)$. Since the value at $t=1$ is assumed to be a coisotropic structure, we get that $\alpha^1+\alpha^2=0$ and $\gamma^1+\gamma^2=0$. The first equation is equivalent to $\lambda^1\colon U\rightarrow \wedge^2(\g)$ being $\h$-invariant while the second equation is equivalent to
\[\phi - \frac{1}{2}\llbracket\lambda^1, \lambda^1\rrbracket - \Alt(\ddr \lambda^1) = 0,\]
i.e. $r = 2\lambda^1 + c$ is a quasi-triangular classical dynamical $r$-matrix.
\end{proof}

\begin{remark}
Strictly speaking, since we consider algebraic functions on $U\subset \h^*$, we only capture dynamical $r$-matrices with a rational dependence on the dynamical parameter. One can also consider trigonometric dependence on the dynamical parameter by considering open subschemes of $H^*$, where $H$ is a Poisson subgroup of $G$ (see \cite{FM} for a related formalism).
\end{remark}

Let us now draw some consequences from this point of view on classical dynamical $r$-matrices. Fix a classical dynamical $r$-matrix giving rise to a 1-shifted Poisson morphism $[U/H]\rightarrow \B G$. The natural projection $U\rightarrow [U/H]$ has an obvious coisotropic (even Lagrangian) structure, hence the composite
\[U\longrightarrow [U/H]\longrightarrow \B G\]
has a 1-shifted coisotropic structure. In particular, by \cref{prop:poissonspaces} we get that $U\times G$ is quasi-Poisson $G$-space. It is easy to see that this quasi-Poisson structure is equivalent to the one given in \cite[Section 2.2]{EE}.

Similarly, interpreting the 1-shifted coisotropic structure on $U\rightarrow \B G$ in terms of a Poisson groupoid over $U$ following \cref{cor:poissongroupoids} we recover dynamical Poisson groupoids introduced in \cite{EV1}.

In a different direction, suppose the classical dynamical $r$-matrix is regular at the origin $0\in\h^*$. In other words, $0\in U$. Then the inclusion of the origin gives a coisotropic morphism $\B H\rightarrow [U/H]$. Therefore, the composite
\[\B H\longrightarrow [U/H]\longrightarrow \B G\]
has a 1-shifted coisotropic structure. Again applying \cref{prop:poissonspaces} we recover a structure of a quasi-Poisson $G$-space on $G/H$ studied in \cite{Lu1} and \cite{KMST}.

\subsection{Construction of classical \texorpdfstring{$r$}{r}-matrices}

\label{sect:rmatrixconstruction}

In this section we will explain how some manipulations with Lagrangian correspondences give rise to classical (dynamical) $r$-matrices.

We begin with the non-dynamical case. Suppose $G$ is a split simple algebraic group, $B_+\subset G$ and $B_-\subset G$ are a pair of opposite Borel subgroups sharing a common maximal torus $H$. The Killing form defines a nondegenerate $G$-invariant pairing on $\g$ and hence a $2$-shifted symplectic structure on $\B G$. The pairing restricts to a nondegenerate pairing on $\h$ and hence gives rise to a $2$-shifted symplectic structure on $\B H$. By \cite[Lemma 3.4]{Sa2} we have Lagrangian correspondences
\[
\xymatrix{
& \B B_{\pm} \ar[dl] \ar[dr] & \\
\B H && \B G
}
\]

Here we set $B_+\rightarrow H$ to be the natural projection and $B_-\rightarrow H$ to be the natural projection post-composed with the inverse on $H$.

Composing Lagrangian correspondences for $B_+$ and $B_-$ we get a Lagrangian correspondence
\[
\xymatrix{
& \B B_+\times_{\B H} \B B_- \ar[dl] \ar[dr] & \\
\B G && \B G
}
\]

If we denote $G^* = B_+\times_H B_-$, then $\B G^*\cong \B B_+\times_{\B H} \B B_-$. Therefore, we get a pair of Lagrangians
\[
\xymatrix{
\B G^* \ar[dr] && \B G \ar[dl] \\
& \B D &
}
\]
where $D = G\times G$ and its Lie algebra $\fd = \g\oplus \g$ is equipped with the difference of the pairings on the two summands.

By construction the two subspaces $\g^*,\g\subset \fd$ are transverse. Indeed, let us split $\b_+\cong \n_+\oplus \h$ and $\b_-\cong \n_-\oplus \h$, where $\n_+$ and $\n_-$ are the corresponding nilpotent subalgebras. Then $\g^*\subset \g\oplus \g$ is given by pairs $(x_+, x_-)$ of elements of $\g$ such that $x_{\pm}\in\b_{\pm}$ and their $\h$-components add up to zero. Then the intersection $\Delta(\g)\cap \g^*$ is clearly zero. But since they are Lagrangian, the intersection is also transverse.

Therefore, the natural basepoint $\pt\rightarrow \B G^*\times_{\B D} \B G$ has a unique $1$-shifted Lagrangian structure and hence we obtain a 2-shifted Lagrangian correspondence
\[
\xymatrix{
& \pt \ar[dl] \ar[dr] & \\
\B G^* \ar[dr] && \B G \ar[dl] \\
& \B D &
}
\]

In particular, we obtain a $1$-shifted coisotropic structure on $\pt\rightarrow \B G$. We are left to show that the underlying $1$-shifted Poisson structure on $\B G$ can be lifted to the natural $2$-shifted symplectic structure on $\B G$ obtained from the pairing on $\g$.

From explicit formulas it is easy to see (e.g. see \cite[Example 2.1.5]{AKS}) that the quasi-Lie bialgebra structure on $\g$ coming from the Manin pair $\g\stackrel{\Delta}\subset \g\oplus \g$ is equivalent to the one given in \cref{prop:forget21BG}. It also follows from the following easy argument on the level of shifted Poisson structures.

\begin{lm}
Consider the $2$-shifted Lagrangian $\Delta\colon \B G\rightarrow \overline{\B G}\times \B G$ given by the diagonal, where $\overline{\B G}$ has the opposite $2$-shifted symplectic structure. The underlying $1$-shifted Poisson structure on $\B G$ is equivalent to one obtained from the $2$-shifted coisotropic structure $\id\colon \B G\rightarrow \B G$.
\end{lm}
\begin{proof}
Consider the induced $2$-shifted coisotropic structure on $\Delta\colon \B G\rightarrow \overline{\B G}\times \B G$. The natural projection $p_2\colon \overline{\B G}\times \B G\rightarrow \B G$ of $2$-shifted Poisson stacks is compatible with the Poisson structure, so the composite
\[\B G\xrightarrow{\Delta} \overline{\B G}\times \B G\xrightarrow{p_2} \B G\]
acquires a 2-shifted coisotropic structure whose underlying $1$-shifted Poisson structure on $\B G$ is equivalent to the one obtained from the diagonal Lagrangian. But the space of $2$-shifted coisotropic structures on the identity $\id\colon \B G\rightarrow \B G$ compatible with the given $2$-shifted Poisson structure on $\B G$ is contractible by \cite[Proposition 4.16]{MS1} which gives the result.
\end{proof}

To summarize, we have constructed a $1$-shifted coisotropic structure on $\pt\rightarrow \B G$ whose underlying $1$-shifted Poisson structure on $\B G$ lifts to a $2$-shifted Poisson structure on $\B G$ given by the Killing form. Therefore, by \cref{prop:constantrmatrix} we obtain a quasi-triangular classical $r$-matrix on $\g$. This is the so-called standard $r$-matrix on a simple Lie algebra.

\begin{remark}
Belavin and Drinfeld \cite{BD} classified all factorizable classical $r$-matrices on a simple Lie algebra in terms of the so-called Belavin--Drinfeld triples $(\Gamma_1, \Gamma_2, \tau)$, where $\Gamma_1, \Gamma_2\subset \Gamma$ are subsets of simple roots and $\tau\colon \Gamma_1\xrightarrow{\sim} \Gamma_2$ is an isomorphism preserving the pairing satisfying a nilpotency condition. In particular, they have constructed the corresponding Manin triples and the above construction can be repeated with those Manin triples.
\end{remark}

One can similarly give a construction of dynamical $r$-matrices. With notations as before, we have a 1-shifted Lagrangian correspondence \cite[Theorem 3.2]{Sa2}
\[
\xymatrix{
& [\b_+/B_+] \ar[dl] \ar[dr] & \\
[\h/H] && [\g/G]
}
\]

Let $\h^{rss}\subset \h$ be the open subscheme of elements which are regular semisimple as elements of $\g$. The projection $[\b_+/B_+]\rightarrow [\h/H]$ becomes an isomorphism over $[\h^{rss}/H]$, so we obtain a 1-shifted Lagrangian correspondence
\[
\xymatrix{
& [\h^{rss}/H] \ar[dl] \ar[dr] & \\
[\h^{rss}/H] && [\g/G]
}
\]
which is a graph of a morphism $[\h^{rss}/H]\rightarrow [\g/G]$. In particular, it carries a 1-shifted coisotropic structure. By \cref{thm:graphcoisotropic} we therefore see that the morphism $[\h^{rss}/H]\rightarrow [\g/G]$ is 1-shifted Poisson. Combining it with \cref{lm:coadjointprojection}, the composite
\[[\h^{rss}/H]\longrightarrow [\g/G]\longrightarrow \B G\]
is also a 1-shifted Poisson morphism and hence by \cref{prop:dynamicalrmatrix} we obtain a classical dynamical $r$-matrix. This is known as the basic rational dynamical $r$-matrix, see \cite[Section 3.1]{ES2}.

\begin{remark}
We may replace adjoint quotients of Lie algebras $[\g/G]$ by adjoint quotients of groups $[G/G]$ in the above construction. Then we recover the basic trigonometric dynamical $r$-matrix.
\end{remark}

\printbibliography

\end{document}